\documentclass[a4paper]{amsart}
\usepackage{amsmath,amsthm,amssymb,amscd,graphicx,xspace,url,mathrsfs,times}

\newcommand{\Z}{\ensuremath{\mathbb{Z}}\xspace}
\newcommand{\Q}{\ensuremath{\mathbb{Q}}\xspace}
\newcommand{\R}{\ensuremath{\mathbb{R}}\xspace}
\newcommand{\C}{\ensuremath{\mathbb{C}}\xspace}
\newcommand{\A}{\ensuremath{\mathbb{A}}\xspace}
\newcommand{\F}{\ensuremath{\mathbb{F}}\xspace}

\newcommand{\g}{\ensuremath{\mathfrak{g}}\xspace}

\newcommand{\m}{\ensuremath{\mathfrak{m}}\xspace}

\newcommand{\q}{\ensuremath{\mathfrak{q}}\xspace}

\newcommand{\OO}{\ensuremath{\mathcal{O}}\xspace}

\newcommand{\comment}[1]{}
\DeclareMathOperator{\Jac}{Jac}

\DeclareMathOperator{\Gal}{Gal}

\DeclareMathOperator{\Hom}{Hom}

\DeclareMathOperator{\Sym}{Sym}
\DeclareMathOperator{\Spec}{Spec}

\newcommand{\rhobar}{\ensuremath{\overline{\rho}}\xspace}
\newcommand{\T}{\ensuremath{\mathbb{T}}\xspace}

\newcommand{\GL}{\ensuremath{\mathrm{GL}}\xspace}
\newcommand{\FF}{\ensuremath{\mathscr{F}}\xspace}
\newcommand{\FFk}{\ensuremath{\mathscr{F}_{k}}\xspace}
\newcommand{\rFFk}{\ensuremath{r_*r^*\mathscr{F}_{k}}\xspace}
\newcommand{\G}{\ensuremath{\mathbb{G}}\xspace}
\newcommand{\cX}{\ensuremath{\mathscr{X}}\xspace}
\newcommand{\cXs}{\ensuremath{\mathscr{X}_{\overline{s}}}\xspace}
\newcommand{\cXe}{\ensuremath{\mathscr{X}_{\overline{\eta}}}\xspace}

\newcommand{\ql}{\ensuremath{U_p/\overline{F^{\times} \cap U_pU^p}}\xspace}
\newcommand{\qlV}{\ensuremath{V_p/\overline{F^{\times} \cap V_pV^p}}\xspace}
\newcommand{\coker}{\ensuremath{\mathrm{coker}}\xspace}
\newcommand{\pf}{\ensuremath{\mathfrak{p}}\xspace}
\newcommand{\ep}{\ensuremath{E}}
\newcommand{\pp}{\ensuremath{v}\xspace}
\newcommand{\qa}{\ensuremath{{\mathfrak{q}_1}}\xspace}
\newcommand{\up}{\ensuremath{U^p(\q)}\xspace}
\newcommand{\qb}{\ensuremath{{\mathfrak{q}_2}}\xspace}
\newcommand{\Mi}{\ensuremath{\mathbb{M}_\q}\xspace}

\newcommand{\red}{\ensuremath{\mathrm{red}}\xspace}
\newtheorem{theorem}[subsection]{Theorem}
\newtheorem{proposition}[subsection]{Proposition}
\newtheorem{corollary}[subsection]{Corollary}
\newtheorem{lemma}[subsection]{Lemma}
\newtheorem{definition}[subsection]{Definition}

\newtheorem*{remark}{Remark}
\mathchardef\mhyphen="2D
\title[Completed cohomology of Shimura curves]{Completed cohomology of Shimura curves and a $p$-adic Jacquet-Langlands correspondence}
\author{James Newton}
\address{DPMMS, University of Cambridge, Centre for Mathematical Sciences, Wilberforce Road, Cambridge CB3 0WA, UK}
\email{jjmn2@cam.ac.uk}

\input xy
\xyoption{all}
\begin{document}
\begin{abstract}
We study indefinite quaternion algebras over totally real fields $F$, and give an example of a cohomological construction of $p$-adic Jacquet-Langlands functoriality using completed cohomology. We also study the (tame) levels of $p$-adic automorphic forms on these quaternion algebras and give an analogue of Mazur's `level lowering' principle. 
\end{abstract}
\maketitle
\section{Introduction}
Our primary goal in this paper is to give some examples of $p$-adic interpolation of Jacquet-Langlands functoriality by studying the completed cohomology \cite{Emint} of Shimura curves. The first example of a $p$-adic interpolation of Jacquet-Langlands functoriality appears in work of Chenevier \cite{MR2111512}, who produces a rigid analytic map between the eigencurve for a definite quaternion algebra over $\Q$ (as constructed in \cite{Bu1}) and the Coleman-Mazur eigencurve for $\GL_2/\Q$ \cite{CM}.

In this paper we will be concerned with the Jacquet-Langlands transfer between automorphic representations of the multiplicative group of quaternion algebras over a totally real field $F$, which are split at one infinite place (so the attached Shimura varieties are one-dimensional). These quaternion algebras will always be split at places dividing $p$. For the purposes of this introduction we will suppose that $F=\Q$. Let $D$ be a quaternion algebra over $\Q$, split at the infinite place and with discriminant $N$, and let $D'$ be a quaternion algebra over $\Q$, split at the infinite place and with discriminant $Nql$ for $q,l$ distinct primes (coprime to $Np$). In this situation there is a rather geometric description of the Jacquet-Langlands transfer from automorphic representations of $(D'\otimes_\Q \A_\Q)^\times$ to automorphic representations of $(D\otimes_\Q \A_\Q)^\times$, arising from the study of the reduction modulo $q$ or $l$ of integral models of Shimura curves attached to $D$ and $D'$ (where we consider Shimura curves attached to $D$ with Iwahori level at $q$ and $l$). For the case $N=1$ and $F=\Q$ this is one of the key results in \cite{Ribet100}, and it was extended to general totally real fields by Rajaei \cite{Raj}. Using the techniques from these papers, together with Emerton's completed cohomology \cite{Emint}, we can describe a $p$-adic Jacquet-Langlands correspondence in a geometric way --- the result obtained is stronger than just the existence of a map between eigenvarieties, which is all that can be deduced from interpolatory techniques as in \cite{MR2111512}.

Our techniques also allow us to prove a $p$-adic analogue of Mazur's principle (as extended to the case of totally real fields by Jarvis \cite{JarMazPrin}), generalising Theorem 6.2.4 of \cite{Emlg}. This should have applications to questions of local-global compatibility at $l \ne p$ in the $p$-adic setting. Having proved a level lowering result for completed cohomology, we can deduce a level lowering result for eigenvarieties. For $F=\Q$ these level lowering results have already been obtained by Paulin \cite{NonComp}, using the results of \cite{Emlg} (in an indirect way, rather than by applying Emerton's construction of the eigencurve).

A related approach to overconvergent Jacquet-Langlands functoriality is pursued in \cite{HIS}, using Stevens's cohomological construction of the eigencurve from overconvergent modular symbols.

We briefly indicate the main results proven in what follows:
\subsection{Level lowering}
We prove an analogue of Mazur's principle for $p$-adic automorphic forms (see Theorems \ref{mazprin} and \ref{mazprinjacquet} for precise statements). Roughly speaking, we show that if a two-dimensional $p$-adic representation $\rho$ of $\Gal(\overline{F}/F)$ ($F$ a totally real field) is unramified at a finite place $\q \nmid p$ and occurs in the completed cohomology of a system of Shimura curves with tame level equal to $U_0(\q)$ at the $\q$-factor, then the system of Hecke eigenvalues attached to $\rho$ occurs in the completed cohomology of Shimura curves with tame level equal to  $\GL_2(\OO_\q)$ at the $\q$-factor.

\subsection{A $p$-adic Jacquet-Langlands correspondence} We show the existence of a short exact sequence of admissible unitary Banach representations of $\GL_2(F\otimes_\Q \Q_p)$ (Theorem \ref{RRES} --- in fact the exact sequence is defined integrally) which encodes a $p$-adic Jacquet-Langlands correspondence from a quaternion algebra $D'/F$, split at one infinite place, with discriminant $\mathfrak{N}\qa\qb$, to a quaternion algebra $D/F$, split at the same infinite place and with discriminant $\mathfrak{N}$. Applying Emerton's eigenvariety construction to this exact sequence gives a map between eigenvarieties (Theorem \ref{ocjl}), i.e. an overconvergent Jacquet-Langlands correspondence in the same spirit as the main theorem of \cite{MR2111512}. As a corollary of these results we also obtain level raising results (Corollary \ref{levelraising}).

\section{Preliminaries}
\subsection{Vanishing cycles}\label{sec:van}
In this section we will follow the exposition of vanishing cycles given in the first chapter of \cite{Raj} - we give some details to fix ideas and notation. Let $\OO$ be a characteristic zero Henselian discrete valuation ring, with residue field $k$ of characteristic $l$, and fraction field $L$. Let $\cX\rightarrow S=\Spec(\OO)$ be a proper, generically smooth curve with reduced special fibre, and define $\Sigma$ to be the set of singular points on the geometric special fibre $\cXs:= \cX\times_S \Spec(\overline{k})$. We also write $\cXe$ for the geometric generic fibre $\cX\times_S \Spec(\overline{L})$. We assume that a neighbourhood of each point $x \in \Sigma$ is \'etale locally isomorphic over $S$ to $\Spec(\OO[t_1,t_2]/(t_1 t_2 - a_x))$ with $a_x$ a non-zero element of $\OO$ with valuation $e_x=v(a_x)>0$.
Define morphisms $i$ and $j$ to be the natural maps:
$$i: \cXs \rightarrow \cX, j: \cXe \rightarrow \cX.$$
For $\FF$ a constructible torsion sheaf on $\cX$, with torsion prime to $l$, the Grothendieck (or Leray) spectral sequence for the composition of the two functors $\Gamma(\cX,-)$ and $j_*$ gives an identification $$R\Gamma(\cXe,j^*\FF)=R\Gamma(\cX,R(j)_*j^*\FF).$$ Applying the proper base change theorem we also have $$R\Gamma(\cX,R(j)_*j^*\FF)\simeq R\Gamma(\cXs,i^*R(j)_*j^*\FF).$$ The natural adjunction $id\implies R(j)_*j^*$ gives a morphism $i^*\FF \rightarrow i^*R(j)_*j^*\FF$ of complexes on $\cXs$. We denote by $R\Phi(\FF)$ the mapping cone of this morphism (the complex of \emph{vanishing cycles}), and denote by $R\Psi(\FF)$ the complex $i^*R(j)_*j^*\FF$ (the complex of \emph{nearby cycles}). There is a distinguished triangle $$i^*\FF \rightarrow R\Psi(\FF) \rightarrow R\Phi(\FF) \rightarrow,$$ whence a long exact sequence of cohomology $$\cdots\rightarrow H^i(\cXs,i^*\FF)\rightarrow H^i(\cXs,R\Psi(\FF))\rightarrow H^i(\cXs,R\Phi(\FF))\rightarrow\cdots$$
As on page 36 of \cite{Raj} there is a \emph{specialisation exact sequence}, with `$(1)$' denoting a Tate twist of the $\Gal(\overline{L}/L)$ action:
\small\begin{equation}\label{speces}
\minCDarrowwidth20pt\begin{CD}0 @>>> H^1(\cXs,i^*\FF)(1) @>>> H^1(\cXe,j^*\FF)(1) @>\beta>> \bigoplus_{x\in\Sigma} R^1\Phi(\FF)_x(1) \\
 @>\gamma>> H^2(\cXs,i^*\FF)(1) @>{sp(1)}>> H^2(\cXe,j^*\FF)(1) @>>> 0.\end{CD}
\end{equation}\normalsize
We can extend the above results to lisse \'etale $\Z_p$-sheaves ($p \ne l$), by taking inverse limits of the exact sequences for the sheaves $\FF/p^i\FF$. Since our groups satisfy the Mittag-Leffler condition we preserve exactness. In particular we obtain the specialisation exact sequence (\ref{speces}) for $\FF$. From now on $\FF$ will denote a $\Z_p$-sheaf, and we define 

$$X(\FF) := \ker(\gamma) = \mathrm{im}(\beta) \subset \bigoplus_{x\in\Sigma} R^1\Phi(\FF)_x(1),$$ so there is a short exact sequence 
$$\minCDarrowwidth20pt\begin{CD}0 @>>> H^1(\cXs,i^*\FF)(1) @>>> H^1(\cXe,j^*\FF)(1) @>>> X(\FF) @>>> 0.\end{CD}$$

As in section 1.3 of \cite{Raj} there is also a cospecialisation sequence
\small\begin{equation}\label{cospeces}
\minCDarrowwidth20pt\begin{CD}0 @>>> H^0(\widetilde{\cXs},R\Psi(\FF)) @>>> H^0(\widetilde{\cXs},\FF) @>{\gamma'}>> \bigoplus_{x\in\Sigma} H^1_x(\cXs,R\Psi(\FF)) \\
 @>{\beta'}>> H^1(\cXs,R\Psi(\FF)) @>>> H^1(\cXs,\FF) @>>> 0,\end{CD}
\end{equation}\normalsize
where we write $\widetilde{\cXs}$ for the normalisation of $\cXs$, and, abusing notation, we retain the same notations $\FF$ (respectively  $R\Psi(\FF)$) for the pullbacks of $\FF$ to $\cXs$ and $\widetilde{\cXs}$ (resp. for the pullback of $R\Psi(\FF)$) to $\widetilde{\cXs}$).  The normalisation map is denoted $r:\widetilde{\cXs}\rightarrow \cXs$.
We can now define 
$$\check{X}(\FF) := \mathrm{im}(\beta'),$$ and apply Corollary 1 of \cite{Raj} to get the following:
\begin{proposition}We have the following diagram made up of two short exact sequences:
$$\xymatrix{& 0\ar[d]\\
&\check{X}(\FF)\ar[d]\\
0\ar[r]&H^1(\cXs, \FF) \ar[r]\ar[d]&H^1(\cXe, \FF)\ar[r]&X(\FF)(-1)\ar[r]&0\\
&H^1(\cXs, r_*r^*\FF)\ar[d]\\
&0}$$
\end{proposition}
Note that the map $\check{X}(\FF)\rightarrow H^1(\cXs,\FF)$ is the one induced by the fact that $$\check{X}(\FF) \subset H^1(\cXs,R\Psi(\FF))$$ is in the image of the injective map $$H^1(\cXs,\FF)\rightarrow H^1(\cXs,R\Psi(\FF)).$$ Finally, as in (1.13) of \cite{Raj} there is an injective map $$\lambda: X(\FF)\rightarrow \check{X}(\FF)$$ coming from the monodromy pairing.

\section{Shimura curves and their bad reduction}
\subsection{Notation}
Let $F/\Q$ be a totally real number field of degree $d$. We denote the infinite places of $F$ by $\tau_1,...,\tau_d$. We will be comparing the arithmetic of two quaternion algebras over $F$. The first quaternion algebra will be denoted by $D$, and we assume that it is split at $\tau_1$ and at all places in $F$ dividing $p$, and non-split at $\tau_i$ for $i \ne 1$ together with a finite set $S$ of finite places. If $d=1$ then we furthermore assume that $D$ is non-split at some finite place (so we avoid working with non-compact modular curves). Fix two finite places $\qa$ and $\qb$ of $F$, which do not divide $p$ and where $D$ is split, and denote by $D'$ a quaternion algebra over $F$ which is split at $\tau_1$ and non-split at $\tau_i$ for $i \ne 1$ together with each finite place in $S \cup \{\qa,\qb\}$. We fix maximal orders $\OO_D$ and $\OO_{D'}$ of $D$ and $D'$ respectively, and an isomorphism $D\otimes_F \A_F^{(\qa \qb)} \cong D'\otimes_F \A_F^{(\qa \qb)}$ compatible with the choice of maximal orders in $D$ and $D'$.

We write $\q$ to denote one of  the places $\qa,\qb$, and denote the completion $F_\q$ of $F$ by $L$, with ring of integers $\OO_\q$ and residue field $k_\q$. Denote $F\cap \OO_\q$ by $\OO_{(\q)}$. We have the absolute Galois group $G_\q=\Gal(\overline{L}/L)$ with its inertia subgroup $I_\q$.

Let $G$ and $G'$ denote the reductive groups over $\Q$ arising from the unit groups of $D$ and $D'$ respectively. Note that $G$ and $G'$ are both forms of $\mathrm{Res}_{F/\Q}(\GL_2/F)$. For $U$ a compact open subgroup of $G(\A_f)$ and $V$ a compact open subgroup of $G'(\A_f)$ we have complex (disconnected) Shimura curves

$$M(U)(\C) = G(\Q)\backslash G(\A_f)\times (\C-\R)/U$$
$$M'(U)(\C) = G'(\Q)\backslash G'(\A_f)\times (\C-\R)/V,$$ where $G(\Q)$ and $G'(\Q)$ act on $\C-\R$ via the $\tau_1$ factor of $G(\R)$ and $G'(\R)$ respectively.

Both these curves have canonical models over $F$, which we denote by $M(U)$ and $M'(U)$. We follow the conventions of \cite{carmauv} to define this canonical model --- see \cite{BDJ} for a discussion of two different conventions used when defining canonical models of Shimura curves.

\subsection{Integral models and coefficient sheaves}\label{subsec:intco}
Let $U \subset G(\A_f)$ be a compact open subgroup (we assume $U$ is a product of local factors over the places of $F$), with $U_\q = U_0(\q)$ (the matrices in $\GL_2(\OO_\q)$ which have upper triangular image in $\GL_2(k_\q)$), and $U$ unramified at all the finite places $v$ where $D$ is non-split (i.e. $U_v=\OO_{D,v}^\times$). Let $\Sigma(U)$ denote the set of finite places where $U$ is ramified. In Theorem 8.9 of \cite{JarMazPrin}, Jarvis constructs an integral model for $M(U)$ over $\OO_{\q}$ which we will denote by $\Mi(U)$. In fact this integral model only exists when $U$ is sufficiently small (a criterion for this is given in \cite[Lemma 12.1]{JarMazPrin}), but this will not cause us any difficulties (see the remark at the end of this subsection). Section 10 of \cite{JarMazPrin} shows that $\Mi(U)$ satisfies the conditions imposed on $\cX$ in section \ref{sec:van} above.

We also have integral models for places dividing the discriminant of our quaternion algebra. Let $V \subset G'(\A_f)$ be a compact open subgroup with $V_\q = \OO_{D',\q}^\times$. We let $\Mi'(V)$ denote the integral model for $M'(V)$ over $\OO_\q$ coming from Theorem 5.3 of \cite{VarII} (this is denoted $\mathfrak{C}$ in \cite{Raj}). This arises from the $\q$-adic uniformisation of $M'(V)$ by a formal $\OO_\q$-scheme, and gives integral models for varying $V$ which are compatible under the natural projections, and again satisfy the conditions necessary to apply section \ref{sec:van}.

We now want to construct sheaves on $\Mi(U)$ and $\Mi'(V)$ corresponding to the possible (cohomological) weights of Hilbert modular forms. Let $k = (k_1,\ldots k_d)$ be a $d$-tuple of integers (indexed by the embeddings $\tau_i$ of $F$ into $\R$ if the reader prefers), all $\ge 2$ and of the same parity. The $d$-tuple of integers $v$ is then characterised by having non-negative entries, at least one of which is zero, with $k+2v=(w,...,w)$ a parallel vector. We will now define $p$-adic lisse \'etale sheaves $\FFk$ on $\Mi(U)$, as in \cite{Car,milne,Sai}. Denote by $Z$ the centre of the algebraic group $G$. Note that $Z \cong\mathrm{Res}_{F/\Q}(\G_{m,F})$. We let $Z_s$ denote the maximal subtorus of $Z$  that is split over $\R$ but which has no subtorus split over $\Q$. It is straightforward to see that $Z_s$ is the subtorus given by the kernel of the norm map $$N_{F/\Q}:\mathrm{Res}_{F/\Q}(\G_{m,F})\rightarrow \G_{m,\Q}.$$
In the notation of \cite[Ch. III]{milne} the algebraic group $G^c$ is the quotient of $G$ by $Z_s$. Take $F'\subset \C$ a number field of finite degree, Galois over $\Q$ and containing $F$ such that $F'$ splits both $D$ and $D'$. Since $F'/\Q$ is normal, $F'$ contains all the Galois conjugates of $F$ and we can identify the embeddings $\{\tau_i:F\rightarrow \C\}$ with the embeddings $\{\tau_i:F\rightarrow F'\}$, via the inclusion $F' \subset \C$. Since $D\otimes_{F,\tau_i} F' \cong M_2(F')$ for each $i$, we obtain representations $\xi_i$ of $D^\times=G(\Q)$ on the space $W_i=F'^2$. Let $\nu:G(\Q)\rightarrow F^\times$ denote the reduced norm. Then $$\xi^{(k)} = \otimes_{i=1}^d (\tau_i\circ \nu)^{v_i}\otimes\Sym^{k_i-2}(\xi_i)$$ is a representation of $G(\Q)$ acting on $W_k=\otimes_{i=1}^d \Sym^{k_i-2}(W_i)$. In fact $\xi^{(k)}$ arises from an algebraic representation of $G$ on the $\Q$-vector space underlying $W_k$ (or rather the associated $\Q$-vector space scheme). For $z \in Z(\Q)=F^\times$ the action of $\xi^{(k)}(z)$ is multiplication by $N_{F/\Q}(z)^{w-2}$, so $\xi^{(k)}$ factors through $G^c$.
The representation $\xi^{(k)}$ gives rise to a representation of $G(\Q_p)$ on $W_k\otimes_\Q \Q_p = W_k \otimes_{F'} (F' \otimes_\Q \Q_p)$, so taking a prime $\pf$ in $F'$ above $p$ and projecting to the relevant factor of $F' \otimes_\Q \Q_p$ we obtain an $F'_\pf$-linear representation $W_{k,\pf}$ of $G(\Q_p)$.

Now we may define a $\Q_p$-sheaf on $M(U)(\C)$:

$$\FF_{k,\Q_p,\C} = G(\Q)\backslash(G(\A_f) \times (\C - \R) \times W_{k,\pf}) / U,$$ where $G(\Q)$ acts on $W_{k,\pf}$ by the representation described above (via its embedding into $G(\Q_p)$), and $U$ acts only on the $G(\A_f)$ factor, by right multiplication. Note that this definition only works because $\xi^{(k)}$ factors through $G^c$.

Choose an $\OO_{F'_\pf}$ lattice $W_{k,\pf}^0$ in $W_{k,\pf}$. For convenience, we will fix a choice of lattice for each $k$. After projection to $G(\Q_p)$, $U$ acts on $W_{k,\pf}$, and if $U$ is sufficiently small it stabilises $W_{k,\pf}^0$. Pick a normal open subgroup $U_1 \subset U$ acting trivially on $W_{k,\pf}^0/p^n$. We now define a $\Z/p^n\Z$-sheaf on $M(U)(\C)$ by
$$\FF_{k,\Z/p^n\Z,\C} = (M(U_1)(\C)\times (W_{k,\pf}^0/p^n)) / (U/U_1).$$
Now on the integral model $\Mi(U)$ we can define an \'etale sheaf of $\Z/p^n\Z$-modules by
$$\FF_{k,\Z/p^n\Z} = (\Mi(U_1)\times (W_{k,\pf}^0/p^n)) / (U/U_1).$$
Taking inverse limits of these $\Z/p^n\Z$-sheaves we get a $\Z_p$- (in fact $\OO_{F'_\pf}$-) lisse \'{e}tale sheaf on $\Mi(U)$ which will be denoted by $\FFk$.

The same construction applies to $\Mi'(V)$, and we also denote the resulting \'{e}tale sheaves on $\Mi'(V)$ by $\FFk$. Note that we chose $F'$ so that it split both $D$ and $D'$ - this allows $F'_\pf$ to serve as a field of coefficients for both $\Mi'(U)$ and $\Mi'(V)$.

From now on we work with a coefficient field $E$. This can be any finite extension of $F'_\pf$. Denote the ring of integers in $E$ by $\OO_\pf$, and fix a uniformiser $\varpi$. The sheaf $\FFk$ on the curves $\Mi(U)$ and $\Mi'(V)$ will now denote the sheaves constructed above with coefficients extended to $\OO_\pf$. 

\begin{remark}Note that in the above discussion we assumed that the level subgroup $U$ was `small enough'. In the applications below we will be taking direct limits of cohomology spaces as the level at $p$, $U_p$, varies over all compact open subgroups of $G(\Q_p)$, so we can always work with a sufficiently small level subgroup by passing far enough through the direct limit.\end{remark}

\subsection{Hecke algebras}\label{sec:hecke}
We will briefly recall the definition of Hecke operators, as in \cite{DT,Raj}. Let $U$ and $U'$ be two (sufficiently small) compact open subgroups of $G(\A_f)$ as in the previous section, and suppose we have $g \in G(\A_f)$ satisfying $g^{-1}U'g \subset U$. Then there is a (finite) \'etale map $\rho_g:M(U')\rightarrow M(U)$, corresponding to right multiplication by $g$ on $G(\A_f)$. Now for $v$ a finite place of $F$, with $D$ unramified at $v$, let $\eta_v$ denote the element of $G(\A_f)$ which is equal to the identity at places away from $v$ and equal to $\begin{pmatrix}1 & 0 \\0 & \varpi_v \end{pmatrix}$ at the place $v$, where $\varpi_v$ is a fixed uniformiser of $\OO_{F,v}$. We also use $\varpi_v$ to denote the element of $G(\A_f)$ which is equal to the identity at places away from $v$ and equal to $\begin{pmatrix}\varpi_v & 0 \\0 & \varpi_v \end{pmatrix}$ at the place $v$. Then the Hecke operator $T_v$ is defined by means of the correspondence on $M(U)$ coming from the maps $$\rho_1,\rho_{\eta_v}:M(U\cap\eta_v U \eta_v^{-1})\rightarrow M(U),$$ whilst $S_v$ is defined by the correspondence $$\rho_1,\rho_{\varpi_v}:M(U)\rightarrow M(U).$$ Note that we are using the notation $T_v$ even when the group $U$ is ramified at $v$. We can make the same definitions for the curve $M'(V)$. 
\begin{definition} The Hecke algebra $\T_k(U)$ is defined to be the $\OO_\pf$-algebra of endomorphisms of $H^1(M(U)_{\overline{F}},\FFk)$ generated by Hecke operators $T_v, S_v$ for finite places $v \nmid p\qa\qb$ such that $U$ and $D$ are both unramified at $v$. The Hecke algebra $\T'_k(V)$ is defined to be the $\OO_\pf$-algebra of endomorphisms of $H^1(M'(V)_{\overline{F}},\FFk)$ generated by Hecke operators $T_v, S_v$ for finite places $v \nmid p$ such that $V$ and $D'$ are both unramified at $v$.
\end{definition}
The Hecke algebras $\T_k(U)$ and $\T'_k(V)$ are semilocal, and their maximal ideals correspond to semisimple mod $p$ Galois representations arising from Hecke eigenforms occurring in $H^1(M(U)_{\overline{F}},\FFk)$ and $H^1(M'(V)_{\overline{F}},\FFk)$ respectively. Given a maximal ideal $\m$ of $\T_k(U)$ in the support of the maximal $\OO_\pf$-torsion free quotient of $H^1(M(U)_{\overline{F}},\FFk)$ we denote the Hecke algebra's localisation at $\m$ by $\T_k(U)_{\rhobar}$, where $\rhobar: \Gal(\overline{F}/F) \rightarrow \GL_2(\overline{\F}_p)$ is the (semisimple) mod $p$ Galois representation attached to $\m$, and more generally for any $\T_k(U)$-module $M$, denote its localisation at $\m$ by $M_{\rhobar}$. If $\rhobar$ is irreducible we say that $\m$ is non-Eisenstein. We apply the same notational convention to $\T'_k(V)$ and modules for this Hecke algebra.

The action of the Hecke operators on the $\OO_\pf$-torsion $H^1(M(U)_{\overline{F}},\FFk)^{\mathrm{tors}}$ is Eisenstein (see \cite[Lemma 4]{DT}), so if $\rhobar$ as above is irreducible then $H^1(M(U)_{\overline{F}},\FFk)_{\rhobar}$ is an $\OO_\pf$-torsion free direct summand of $H^1(M(U)_{\overline{F}},\FFk)$, with a faithful action of $\T_k(U)_{\rhobar}$. The same considerations apply to the Hecke algebra $\T_k'(V)$.

Note that if $k=(2,...,2)$ then $H^1(M(U)_{\overline{F}},\FFk)=H^1(M(U)_{\overline{F}},\OO_\pf)$ is already $\OO_\pf$-torsion free.

\subsection{Towards Mazur's principle}\label{mazprinclass}
In this subsection and the next we summarise the results of \cite{JarMazPrin} and \cite{Raj}. First we fix a (sufficiently small) compact open subgroup $U \subset G(\A_f)$ such that $U_\q = \GL_2(\OO_\q)$. For brevity we denote the compact open subgroup $U\cap U_0(\q)$ by $U(\q)$. 

We apply the theory of Section \ref{sec:van} to the curve $\Mi(U(\q))$ and the $\Z_p$-sheaf given by $\FFk$, for some weight vector $k$, to obtain $\OO_\pf$-modules $X_\q(U(\q),\FFk)$ and $\check{X}_\q(U(\q),\FFk)$. We have short exact sequences
\small\begin{equation}\label{ccses}\minCDarrowwidth10pt\begin{CD}0 @>>> H^1(\Mi(U(\q))_{\overline{s}}, \FFk) @>>> H^1(\Mi(U(\q))_{\overline{\eta}}, \FFk) @>>> X_\q (U(\q),\FFk)(-1) @>>> 0,\end{CD}\end{equation}
\begin{equation}\label{cccoses}\minCDarrowwidth10pt\begin{CD}0@>>>\check{X}_\q(U(\q),\FFk)@>>> H^1(\Mi(U(\q))_{\overline{s}}, \FFk)@>>> H^1(\Mi(U(\q))_{\overline{s}}, r_*r^*\FFk)@>>> 0.\end{CD}\end{equation}\normalsize
The above short exact sequences are equivariant with respect to the $G_\q$ action. Moreover, as discussed in \cite[\S 15]{JarMazPrin}, the correspondences of section \ref{sec:hecke} extend to correspondences on the integral models for our Shimura curves, so for $v \nmid p\q$ we may define actions of the Hecke operators $T_v, S_v$ on all the groups appearing in the sequences (\ref{ccses}), (\ref{cccoses}). The sequences are then equivariant with respect to the actions of these Hecke operators. For an irreducible $\rhobar$ associated with a maximal ideal $\m$ of $\T_k(U(\q))$ we can localise at $\m$ to get short exact sequences of $\T_k(U(\q))_{\rhobar}$-modules. The module $X_\q (U(\q),\FFk)$ is always $\OO_\pf$-torsion free (it is a defined as a submodule of a visibly $\OO_\pf$-torsion free module), and the torsion part of $H^1(\Mi(U(\q))_{\overline{s}}, r_*r^*\FFk)$ is Eisenstein (by the same argument as for $H^1(M(U)_{\overline{F}},\FFk)$), so localising at $\m$ actually gives short exact sequences of $\OO_\pf$-torsion free $\T_k(U(\q))_{\rhobar}$-modules.

The sequence (\ref{cccoses}) underlies the proof of Mazur's principle (see \cite{JarMazPrin} and Theorem \ref{mazprin} below).

\subsection{Ribet-Rajaei exact sequence}\label{sec:RR}
We also apply the theory of Section \ref{sec:van} to the curve $\Mi'(V)$ and the sheaf $\FFk$, where $V \subset G'(\A_f)$ is a (sufficiently small) compact open subgroup which is unramified at $\q$. We obtain $\OO_\pf$-modules $Y_\q(V,\FFk)$ and $\check{Y}_\q(V,\FFk),$ where $Y$ and $\check{Y}$ correspond to $X$ and $\check{X}$ respectively. As in equation 3.2 of \cite{Raj} these modules lie in ($G_\q$ and Hecke equivariant) short exact sequences
\small$$\minCDarrowwidth10pt\begin{CD}0@>>>H^1(\Mi'(V)_{\overline{s}}, \FFk) @>>>H^1(\Mi'(V)_{\overline{\eta}}, \FFk)@>>>Y_\q(V,\FFk)(-1)@>>>0,\end{CD}$$
$$\minCDarrowwidth10pt\begin{CD}0@>>>\check{Y}_\q(V,\FFk)@>>>H^1(\Mi'(V)_{\overline{s}}, \FFk)@>>>H^1(\Mi'(V)_{\overline{s}}, r_*r^*\FFk)@>>>0.\end{CD}$$\normalsize

As in the previous section, these sequences may be localised at non-Eisenstein ideals of $\T'_k(V)$ to give short exact sequences of $\OO_\pf$-torsion free modules. It is remarked in section 3 of \cite{Raj} that in fact $H^1(\Mi'(V)_{\overline{s}}, r_*r^*\FFk)=0$, since the irreducible components of the special fibre of $\Mi'(V)$ are rational curves, so we have $$\check{Y}_\q(V,\FFk)\cong H^1(\Mi'(V)_{\overline{s}}, \FFk).$$

We can now relate our constructions for the two quaternion algebras $D$ and $D'$. This requires us to now distinguish between the two places $\qa$ and $\qb$. We make some further assumptions on the levels $U$ and $V$. First we assume that $U_\qa = \OO_{D,\qa} \cong \GL_2(\OO_\qa)$ and $U_\qb = \OO_{D,\qb} \cong \GL_2(\OO_\qb)$, and similarly that $V$ is also maximal at both $\qa$ and $\qb$. We furthermore assume that the factors of $U$ away from $\qa$ and $\qb$, $U^{\qa \qb}$, match with $V^{\qa\qb}$ under the fixed isomorphism $D\otimes_F \A_{F,f}^{(\qa \qb)} \cong D'\otimes_F \A_{F,f}^{(\qa \qb)}$. We will use the notations $U(\qa), U(\qb), U(\qa\qb)$ to denote $U\cap U_0(\qa), U\cap U_0(\qb), U\cap U_0(\qa)\cap U_0(\qb)$ respectively.

We fix a Galois representation $$\rhobar: \Gal(\overline{F}/F) \rightarrow \GL_2(\overline{\F}_p),$$ satisfying two assumptions. Firstly, we assume that $\rhobar$ arises from a Hecke eigenform in $H^1(M'(V)_{\overline{F}},\FFk)$. Secondly, we assume that $\rhobar$ is irreducible. Equivalently, the corresponding maximal ideal $\m'$ in $\T'_k(V)$ is non-Eisenstein. 

There is a surjection $$\T_k(U(\qa\qb)) \rightarrow \T'_k(V)$$ arising from the classical Jacquet--Langlands correspondence, identifying $\T'_k(V)$ as the $\qa\qb$-new quotient of $\T_k(U(\qa\qb))$. Hence there is a (non-Eisenstein) maximal ideal of $\T_k(U(\qa\qb))$ associated to $\rhobar$ (in fact this can be established directly from the results contained on pg. $55$ of \cite{Raj}, but for expository purposes it is easier to rely on the existence of the Jacquet--Langlands correspondence). We can therefore regard all $\T'_k(V)_{\rhobar}$-modules as $\T_k(U(\qa\qb))_{\rhobar}$-modules. Now \cite[Theorem 3]{Raj} implies the following: 

\begin{theorem}\label{RRESclass}There are $\T_k(U(\qa\qb))_{\rhobar}$-equivariant short exact sequences
$$\minCDarrowwidth15pt\begin{CD}0@>>> Y_\qa(V,\FFk)_{\rhobar} @>>>X_\qb(U(\qa\qb),\FFk)_{\rhobar}@>{i^\dagger}>>X_\qb(U(\qb),\FFk)^{\oplus 2}_{\rhobar}@>>>0,\end{CD}$$
$$\minCDarrowwidth15pt\begin{CD}0@>>>\check{X}_\qb(U(\qb),\FFk)^{\oplus 2}_{\rhobar}@>{i}>>\check{X}_\qb(U(\qa\qb),\FFk)_{\rhobar} @>>>\check{Y}_\qa(V,\FFk)_{\rhobar}@>>>0.\end{CD}$$\end{theorem}

In the theorem, the maps labelled $i$ and $i^\dagger$ are the natural `level raising' map and its adjoint. To be more precise $i=(i_1,i_{\eta_\q})$, where $i_g$ is the map coming from the composition of the natural map $$H^1(M(U)_{\overline{F}},\FFk) \rightarrow H^1(M(U(\q))_{\overline{F}},\rho_{g}^*\FFk)$$ with the map coming from the sheaf morphism $\rho_{g}^*\FFk \rightarrow \FFk$ as in \cite[pg. 449]{DT}. The map $i^\dagger=(i_1^\dagger,i_{\eta_\q}^\dagger)$, where $i_g^\dagger$ is the map coming from the composition of the map arising from the sheaf morphism $\FFk \rightarrow \rho_g^*\FFk$ (again see \cite[pg. 449]{DT}) with the map $$H^1(M(U(\q))_{\overline{F}},\rho_g^*\FFk) \rightarrow H^1(M(U)_{\overline{F}},\FFk)$$ arising from the trace map $\rho_{g*}\rho_g^*\FFk\rightarrow \FFk$.

We think of the short exact sequences in the above Theorem as a geometric realisation of the Jacquet--Langlands correspondence between automorphic forms for $G'$ and for $G$. The proof of the above theorem relies on relating the spaces involved to the arithmetic of a third quaternion algebra $\overline{D}$ (with associated reductive algebraic group $\overline{G}/\Q$) which is non-split at all the infinite places of $F$ and at the places in $S \cup \{\qb\}$. However, we will just focus on the groups $G'$ and $G$, and view $\overline{G}$ as playing an auxiliary role.  

\section{Completed cohomology of Shimura curves}

\subsection{Completed cohomology}\label{subsec:ccoho}
We now consider compact open subgroups $U^p$ of $G(\A_f^{(p)})$, with factor at $v$ equal to $\OO_{D_v}^\times$ at all places $v$ of $F$ where $D$ is non-split. Then we have the following definitions, as in \cite{Emint,Emlg}:
\begin{definition}
We write $H^n_D(U^p,\FFk/\varpi^s)$ for the $\OO_\pf$ module $$\varinjlim_{U_p} H^n(M(U_p U^p)_{\overline{F}},\FFk/\varpi^s),$$ which has a continuous action of $G_F$. If $\FFk = \OO_\pf$ then this module also has a smooth action of $G(\Q_p)=\prod_{\pp | p} \GL_2(F_\pp)$.

We then define a $\varpi$-adically complete $\OO_\pf$-module $$\widetilde{H}^n_D(U^p,\FFk) := \varprojlim_s H^n_D(U^p,\FFk/\varpi^s),$$ which again has a continuous action of $G_F$. If $\FFk = \OO_\pf$ this module has a continuous action of $G(\Q_p)$
\end{definition}

\begin{lemma}
The $\OO_\pf$-module $\widetilde{H}^n_D(U^p,\OO_\pf)$ is a $\varpi$-adically admissible $G(\Q_p)$ representation over $\OO_\pf$ (in the sense of \cite[Definition 2.4.7]{Emordone}).
\end{lemma}
\begin{proof}
This follows from \cite[Theorem 2.2.11]{Emint}, taking care to remember the $\OO_\pf$-integral structure.
\end{proof}

Write $\widetilde{H}^n_D(U^p,\FFk)_{\ep}$ for the admissible continuous $E$-Banach $G(\Q_p)$-representation $\widetilde{H}^n_D(U^p,\FFk)\otimes_{\OO_\pf}\ep$ --- we refer to \cite[Lemma 2.2.10]{Emint} for an explanation of the $G(\Q_p)$ action on these spaces. If $U_p \subset G(\Q_p)$ is a compact open subgroup which stabilises the lattice $W_{k,\pf}^0$ then $U_p$ acts on the module $\widetilde{H}^n_D(U^p,\FFk)$. Also, $\widetilde{H}^n_D(U^p,\FFk)_{\ep}$ is `independent of the choice of lattice $W_{k,\pf}^0$' in the sense that there are canonical isomorphisms between the spaces defined using different lattices (see \cite[Lemma 2.2.8]{Emint})

We will usually denote $\widetilde{H}^n_D(U^p,\OO_\pf)_{\ep}$ by $\widetilde{H}^n_D(U^p,E)$.
\begin{lemma}\label{hhat}
There is a canonical isomorphism $$\widetilde{H}^n_D(U^p,\OO_\pf) = \varprojlim_s H^n_D(U^p,\OO_\pf)/\varpi^sH^n_D(U^p,\OO_\pf),$$ compatible with $G_F$ and $G(\Q_p)$ actions.
\end{lemma}
\begin{proof}Since $M(U_p U^p)_{\overline{F}}$ is proper of dimension $1$, this follows from Corollary 2.2.27 of \cite{Emint} \end{proof}
\begin{remark}Note that in the notation of \cite{Emint} the above lemma says that $\widetilde{H}^n_D(U^p) = \widehat{H}^n_D(U^p)$.\end{remark}
The following lemma allows us to restrict to considering completed cohomology with trivial coefficients.
\begin{lemma}\label{changeweight}There is a canonical $G(\Q_p)$ and $G_F$ equivariant isomorphism
$$\widetilde{H}^n_D(U^p,\FFk)_{\ep} \cong \widetilde{H}^n_D(U^p,\OO_\pf)\otimes_{\OO_\pf} W_{k,\pf}.$$
\end{lemma}
\begin{proof}This follows from Theorem 2.2.17 of \cite{Emint}.
\end{proof}

\begin{proposition}Let $\overline{F^{\times,+}}$ be the closure of the totally positive elements of $F^\times$ in $\A_F^{f,\times}$. There is an isomorphism between $\widetilde{H}^0_D(U^p,\OO_\pf)$ and the space of continuous functions $$\mathscr{C}(\overline{F^{\times,+}} \backslash \A_F^{f,\times} / \det(U^p), \OO_\pf),$$ with $G(\Q_p)$ action induced by $g$ acting on $\A_F^{f,\times}$ as multiplication by $\det(g)$.\end{proposition}
\begin{proof}This follows from the description of the connected components of $M(U_p U^p)(\C)$ given by \cite[\S 1.2]{carmauv}.\end{proof}

\begin{proposition}The space $\widetilde{H}^2_D(U^p,\OO_\pf)$ is equal to $0$. \end{proposition}
\begin{proof}Exactly as for Proposition 4.3.6 of \cite{Emint} \end{proof}

Let $\Sigma(U^p)$ denote the set of primes $v$ in $F$ such that $U^p$ and $D$ are both unramified at $v$. If $v \in \Sigma(U^p)$ then there are Hecke operators $T_v, S_v$ acting on the spaces $H^1_D(U^p,\OO_\pf/\varpi^s)$, defined as in Section \ref{sec:hecke}. By continuity this extends to an action of the Hecke operators on the space $\widetilde{H}^1_D(U^p,\OO_\pf)$. We make the following definition: 
\begin{definition} The Hecke algebra $\T(U^p)$ is defined to be $\varprojlim_{U_p}\T_{(2,...,2)}(U^p U_p)$. We endow it with the projective limit topology given by taking the $\varpi$-adic topology on each finitely generated $\OO_\pf$-algebra $\T_{(2,...,2)}(U^p U_p)$.
\end{definition}
The topological $\OO_\pf$-algebra $\T(U^p)$ acts faithfully on $\widetilde{H}^1_D(U^p,\OO_\pf)$, and (viewed as a subalgebra of the topological algebra of continuous endomorphisms of this space) it is topologically generated by the endomorphisms of $\widetilde{H}^1_D(U^p,\OO_\pf)$ given by the Hecke operators $T_v, S_v$ for primes $v \nmid p$ such that $U^p$ and $D$ are both unramified at $v$. Note that the isomorphism of Lemma  \ref{changeweight} intertwines the action of the Hecke operators on $\widetilde{H}^1_D(U^p,\FF_k)_E$ (defined via the weight $k$ action at finite level) with the action of the Hecke operators on $\widetilde{H}^1_D(U^p,\OO_\pf)$. Hence $\widetilde{H}^1_D(U^p,\FF_k)_E$ may be naturally viewed as a faithful $\T(U^p)$-module. 

\begin{definition}An ideal $I$ of $\T(U^p)$ is \emph{Eisenstein} if the map $\lambda: T(U^p) \rightarrow T(U^p)/I$ satisfies $\lambda(T_q) = \epsilon_1(q) + \epsilon_2(q)$ and $\lambda(qS_q) = \epsilon_1(q)\epsilon_2 (q)$ for all $q \notin \{p\}\cup\Sigma(U^p)$, for some characters $\epsilon_1,\epsilon_2 : \Z_{p\Sigma(U^p)}^\times \rightarrow T(U^p)/I$. \end{definition}

The Hecke algebra $\T(U^p)$ is semilocal and Noetherian, with maximal ideals corresponding to mod $p$ Galois representations arising from Hecke eigenforms in (an extension of scalars of) $H^1_D(U^p,\OO_\pf/\varpi)$. Given a maximal ideal $\m$ of $\T(U^p)$ we denote the Hecke algebra's localisation at $\m$ by $\T(U^p)_{\rhobar}$, where $\rhobar: \Gal(\overline{F}/F) \rightarrow \GL_2(\overline{\F}_p)$ is the semisimple mod $p$ Galois representation attached to $\m$, and more generally for any $\T(U^p)$-module $M$, denote its localisation at $\m$ by $M_{\rhobar}$. We now fix such a $\rhobar$, and assume that it is irreducible. We have that $\widetilde{H}^1_D(U^p,\FFk)_{\rhobar,E}$ is a faithful $\T(U^p)_{\rhobar}$-module. Moreover, Lemma \ref{goodinvariants} allows us to identify $\T(U^p)_{\rhobar}$ with the inverse limit $\varprojlim_{U_p}\T_k(U^p U_p)_{\rhobar}$ (where the limit is taken over $U_p$ small enough such that $\rhobar$ arises from a maximal ideal of $\T_k(U^p U_p)$).

There is a deformation $\rho_{\rhobar,U^p}^m$ of $\rhobar$ to $\T(U^p)_{\rhobar}$ (taking a limit, over compact open subgroups $U_p$ of $G(\Q_p)$, of the deformations to $\T_{(2,...,2)}(U^pU_p)_{\rhobar}$). For a closed point $\mathscr{P} \in \Spec(\T(U^p)_{\rhobar}[\frac{1}{p}])$ we denote by $\rho(\mathscr{P})$ the attached Galois representation, defined over the characteristic $0$ field $k(\mathscr{P})$.

\begin{definition}Given a $\T(U^p)_{\rhobar}$-module $M$ and a closed point $\mathscr{P}$ of $\Spec(\T(U^p)_{\rhobar}[\frac{1}{p}])$, we denote by $M[\mathscr{P}]$ the $k(\mathscr{P})$-vector space $$\{m \in M\otimes_{\OO_\pf}k(\mathscr{P}) : Tm = 0 \hbox{ for all } T \in \mathscr{P}\}.$$ Suppose we have a homomorphism $\lambda: \T(U^p)_{\rhobar}\rightarrow E'$, where $E'/E$ is a finite field extension. We say that the system of Hecke eigenvalues $\lambda$ \emph{occurs in} a $\T(U^p)_{\rhobar}$-module $M$ if $M\otimes_{\OO_\pf} E'$ contains a non-zero element where $\T(U^p)$ acts via the character $\lambda$. In other words, $\lambda$ occurs in $M$ if and only if $M[\ker(\lambda)] \ne 0$. We say that $\lambda$ is Eisenstein whenever $\ker{\lambda}$ is.\end{definition}

We recall that after localising at a non-Eisenstein ideal completed cohomology is very simply related to the cohomology of our Shimura curves at finite levels.
\begin{definition}
If $z \in \C - \R$ then let $G(\Q)_z$ denote the stabiliser of $z$ in $G(\Q)$ (recall that $G(\Q)$ acts on $\C - \R$ via the $\tau_1$ factor of $G(\R)$. We say that a compact open subgroup $U$ of $G(\A_f)$ is \emph{neat} if the intersection $G(\Q)_z \cap gUg^{-1}=\{1\}$ for each $z \in \C - \R$ and each $g \in G(\A_f)$. This condition ensures that $G(\Q)$ acts on $(\C-\R) \times G(\A_f) / U$ without fixed points. We can make the same definition for compact open subgroups $V$ of $G'(\A_f)$.
\end{definition}
\begin{lemma}\label{goodinvariants}
Suppose that $U_p$ is a compact open subgroup of $G(\Q_p)$, and $U_p U^p$ is neat. Then the natural map $H^1(M(U_p U^p)_{\overline{F}},\FFk/\varpi^s)_{\rhobar}\rightarrow (H^1_D(U^p,\FFk/\varpi^s)_{\rhobar})^{U_p}$ is an isomorphism. Also, the natural map $H^1(M(U_p U^p)_{\overline{F}},\FFk)_{\rhobar}\rightarrow (\widetilde{H}^1_D(U^p,\FFk)_{\rhobar})^{U_p}$ is an isomorphism.
\end{lemma}
\begin{proof}
The first claim follows from an application of the Hochschild-Serre spectral sequence, as in \cite[Lemma 5.3.8]{Emlg}. The second claim follows from the first by passing to the inverse limit over $s$.
\end{proof}
We can now generalise \cite[Proposition 5.5.3]{Emlg}. Denote by $X$ the $\OO_\pf$-module \[\Hom_{\T(U^p)_{\rhobar}[G_F]}(\rho_{\rhobar,U^p}^m,\widetilde{H}^1_D(U^p,\OO_\pf)_{\rhobar}).\] There is a natural isomorphism
$$X \cong ((\rho_{\rhobar,U^p}^m)^\vee\otimes_{\T(U^p)_{\rhobar}}\widetilde{H}^1_D(U^p,\OO_\pf)_{\rhobar})^{G_F},$$ so $X$ can be viewed as a closed and saturated subrepresentation of $(\rho_{\rhobar,U^p}^m)^\vee\otimes_{\T(U^p)_{\rhobar}}\widetilde{H}^1_D(U^p,\OO_\pf)_{\rhobar}$. Hence by \cite[Proposition 2.4.13]{Emordone}, $X$ is a $\varpi$-adically admissible $G(\Q_p)$-representation over $\OO_\pf$.
\begin{proposition}\label{galfac}
The natural evaluation map $$ev:\rho_{\rhobar,U^p}^m \otimes_{\T(U^p)_{\rhobar}} X \rightarrow \widetilde{H}^1_D(U^p,\OO_\pf)_{\rhobar}$$ is an isomorphism.\end{proposition}
\begin{proof}
First let $\mathscr{P} \in \Spec(\T(U^p)_{\rhobar}[\frac{1}{p}])$ be a prime ideal of the Hecke algebra arising from a Hecke eigenform in $$H^1(M(U^pU_p),\OO_\pf)_{\rhobar}\otimes_{\OO_\pf}k(\mathscr{P}).$$ 
Then the image of the map $$\rho(\mathscr{P}) \otimes_{k(\mathscr{P})}X[\mathscr{P}] \rightarrow \widetilde{H}^1_D(U^p,\OO_\pf)_{\rhobar}\otimes_{\OO_\pf} k(\mathscr{P})$$ induced by $ev$ contains $H^1(M(U^pU_p)_{\rhobar},\OO_\pf)
[\mathscr{P}].$ From this, it is straightforward to to see that the image of the map $$ev_E: E \otimes_{\OO_\pf}\rho_{\rhobar,U^p}^m \otimes_{\T(U^p)_{\rhobar}} X \rightarrow \widetilde{H}^1_D(U^p,E)_{\rhobar}$$ is dense, since it will contain the dense subspace $H^1_D(U^p,E)_{\rhobar}$. Now since $X$ is a $\varpi$-adically admissible $G(\Q_p)$-representation over $\OO_\pf$, \cite[Proposition 3.1.3]{Emlg} implies that the image of $ev_E$ is closed, hence $ev_E$ is surjective (this also follows from the theory of \cite{STei2}, see Proposition 3.1 and Theorem 3.5 of that paper). We now note that the evaluation map $$\rhobar \otimes_{k(\pf)} \Hom_{G_F}(\rhobar,H^1_D(U^p,\OO_\pf/\varpi)_{\rhobar}) \rightarrow H^1_D(U^p,\OO_\pf/\varpi)_{\rhobar}$$ is injective, by the irreducibility of $\rhobar$. We conclude the proof of the proposition by applying the following lemma, with $A=\T(U^p)_{\rhobar}$. Note that as in \cite[Proposition 5.3.5]{Emlg}, $X$ and $\widetilde{H}^1_D(U^p,\OO_\pf)_{\rhobar}$ are in fact $\varpi$-adically admissible representations of $G(\Q_p)$ over $\T(U^p)_{\rhobar}$.
\end{proof}
\begin{lemma}Let $\pi_1$ and $\pi_2$ be two $\varpi$-adically admissible $G(\Q_p)$ representations over a complete local Noetherian $\OO_\pf$-algebra $A$ (with maximal ideal \m), with $\pi_2$ torsion free as an $\OO_\pf$-module. Suppose $f: \pi_1 \rightarrow \pi_2$ is a continuous $A[G(\Q_p)]$-linear morphism,  satisfying \begin{enumerate} \item The induced map $\pi_1 \otimes_{\OO_\pf} E \rightarrow \pi_2 \otimes_{\OO_\pf} E$ is a surjection
\item The induced map $(\pi_1/\varpi \pi_1)[\m] \rightarrow (\pi_2/\varpi \pi_2)[\m]$ is an injection.\end{enumerate}
Then $f$ is an isomorphism.
\end{lemma}
\begin{proof}
First we show that $f$ is an injection. Denote the kernel of $f$ by $K$. It is a $\varpi$-adically admissible $G(\Q_p)$ representation over $A$, so in particular $$K/\varpi K=\bigcup_{i\ge 1} (K/\varpi K)[\m^i].$$ Since $\pi_2$ is $\OO_\pf$-torsion free, the image of $f$ is also $\OO_\pf$-torsion free, so there is an exact sequence
$$0 \rightarrow K/\varpi K \rightarrow \pi_1/\varpi \pi_1 \rightarrow \pi_2/\varpi \pi_2.$$
Our second assumption on $f$ hence implies that $(K/\varpi K)[\m] = 0$, so $(K/\varpi K)[\m^i] = 0$ for all $i$. Therefore $K/\varpi K = 0$, which shows that $K=0$, since $K$ is $\varpi-$adically separated. We now know that $f$ is an injection, and our first assumption on $f$ implies that it has $\OO_\pf$-torsion cokernel $C$. Since $\pi_2$ is $\OO_\pf$-torsion free there is an exact sequence
$$0 \rightarrow C[\varpi] \rightarrow \pi_1/\varpi \pi_1 \rightarrow \pi_2/\varpi \pi_2.$$
Since $\pi_1/\varpi \pi_1=\bigcup_{i\ge 1} (\pi_1/\varpi \pi_1) [\m^i]$ we have $C[\varpi] =\bigcup_{i\ge 1} C[\varpi][\m^i],$ but the second assumption on $f$ implies that $C[\varpi][\m]=0$. Therefore $C[\varpi]=0$, and since $C$ is $\OO_\pf$-torsion we have $C=0$. Therefore $f$ is an isomorphism.
\end{proof}

\subsection{Completed vanishing cycles}
We now assume that our tame level $U^p$ is unramified at the prime $\q$, and let $\up$ denote the tame level $U^p \cap U_0(\q)$. Fix a non-Eisenstein maximal ideal $\m$ of $\T(\up)$, with associated mod $p$ Galois representation $\rhobar$. We are going to `$\varpi$-adically complete' the constructions of subsection \ref{mazprinclass}.

\begin{definition}\label{smoothstuff}
For $\FF$ equal to either $\FFk$ or $r_*r^*\FFk$ we define $\OO_\pf$-modules
$$H^1_{\red}(\up,\FF)_{\rhobar} = \varinjlim_{U_p} H^1(\Mi(U_p\up)_{\overline{s}}, \FF)_{\rhobar},$$ $$X_\q (\up,\FFk)_{\rhobar}=\varinjlim_{U_p}X_\q (U_p\up,\FFk)_{\rhobar},$$ and $$\check{X}_\q(\up,\FFk)_{\rhobar}=\varinjlim_{U_p}\check{X}_\q(U_p\up,\FFk)_{\rhobar}.$$
\end{definition}
In this definition, the limits are taken over those $U_p$ which are sufficiently small that $\rhobar$ arises from a maximal ideal of $\T_k(U_p\up)$.
\begin{proposition}\label{smoothses}
Equation (\ref{ccses}) induces a short exact sequence, equivariant with respect to Hecke and $G_\q$ actions:
$$\minCDarrowwidth10pt\begin{CD}0@>>>H^1_{\red}(\up,\FFk)_{\rhobar} @>>>H^1_D(\up,\FFk)_{\rhobar}@>>>X_\q (\up,\FFk)(-1)_{\rhobar}@>>>0.\end{CD}$$
Similarly, equation (\ref{cccoses}) induces 
$$\minCDarrowwidth10pt\begin{CD}0@>>>\check{X}_\q(\up,\FFk)_{\rhobar}@>>>H^1_{\red}(\up,\FFk)_{\rhobar}@>>>H^1_{\red}(\up,r_*r^*\FFk)_{\rhobar}@>>>0.\end{CD}$$
When $\FFk = \OO_\pf$ all the terms in these sequences have smooth actions of $G(\Q_p)$, and the sequences are equivariant with respect to this action. In any case, after tensoring with $E$ over $\OO_\pf$ the terms in these sequences acquire actions of $G(\Q_p)$ and again the sequences are equivariant with respect to this action. 
\end{proposition}
\begin{proof}
This just follows from taking the direct limits over $U_p$ of the relevant exact sequences.
\end{proof}
\begin{definition}\label{completedstuff}
We define $\T(\up)_{\rhobar}$-modules (again $\FF$ equals either $\FFk$ or $r_*r^* \FFk$)
$$\widetilde{H}^1_\red(\up,\FF)_{\rhobar} = \varprojlim_s H^1_{\red}(\up,\FF)_{\rhobar}/\varpi^sH^1_{\red}(\up,\FF)_{\rhobar},$$ $$\widetilde{X}_\q(\up,\FFk)_{\rhobar} = \varprojlim_s X_\q(\up,\FFk)_{\rhobar}/\varpi^sX_\q (\up,\FFk)_{\rhobar}$$ and $$\widetilde{\check{X}}_\q(\up,\FFk)_{\rhobar} = \varprojlim_s \check{X}_\q(\up,\FFk)_{\rhobar}/\varpi^s\check{X}_\q(\up,\FFk)_{\rhobar}.$$
\end{definition}
We have made these definitions after localising at non-Eisenstein maximal ideals to avoid issues of $\OO_\pf$-torsion in any of our modules at finite level. After $\varpi$-adically completing the short exact sequences of Proposition \ref{smoothses}, we obtain
\begin{proposition}\label{prop:es}
\sloppy We have short exact sequences, equivariant with respect to $\T(\up)_{\rhobar}$ and $G_\q$ actions:
$$\minCDarrowwidth10pt\begin{CD}0@>>>\widetilde{H}^1_{\red}(\up,\FFk)_{\rhobar} @>>>\widetilde{H}^1_D(\up,\FFk)_{\rhobar}@>>>\widetilde{X}_\q(\up,\FFk)_{\rhobar}(-1)@>>>0,\end{CD}$$
$$\minCDarrowwidth10pt\begin{CD}0@>>>\widetilde{\check{X}}_\q(\up,\FFk)_{\rhobar}@>>>\widetilde{H}^1_{\red}(\up,\FFk)_{\rhobar}@>>>\widetilde{H}^1_{\red}(\up,r_*r^*\FFk)_{\rhobar}@>>>0.\end{CD}$$
When $\FFk = \OO_\pf$ all the terms in these sequences have continuous actions of $G(\Q_p)$, and the sequences are short exact sequences of $\varpi$-adically admissible $G(\Q_p)$-representations over $\OO_\pf$. In any case, after tensoring with $E$ over $\OO_\pf$ the terms in these sequences acquire actions of $G(\Q_p)$ and again the sequences are equivariant with respect to this action. 
\end{proposition}
\begin{proof}
The existence of a short exact sequence of $\OO_\pf$-modules follows from the fact that at finite levels localising the short exact sequences of  Proposition \ref{smoothses} at non-Eisenstein ideals gives short exact sequences of $\OO_\pf$-torsion free modules. Hence $\varpi$-adically completing preserves exactness. The fact that we obtain a short exact sequence of $\varpi$-adically admissible $\OO_\pf[G(\Q_p)]$-modules when $\FFk=\OO_\pf$ follows from Proposition 2.4.4 of \cite{Emordone}.
\end{proof}
It is now straightforward to deduce the analogue of Lemma \ref{changeweight}:
\begin{lemma}There are canonical $G(\Q_p)$, $G_\q$ and Hecke equivariant isomorphisms
$$\widetilde{H}^1_{\red}(\up,\FFk)_{\rhobar,E} \cong \widetilde{H}^1_{\red}(\up,\OO_\pf)_{\rhobar}\otimes_{\OO_\pf} W_{k,\pf},$$
$$\widetilde{\check{X}}_\q(\up,\FFk)_{\rhobar,E} \cong \widetilde{\check{X}}_\q(\up,\OO_\pf)_{\rhobar}\otimes_{\OO_\pf} W_{k,\pf},$$
$$\widetilde{X}_\q(\up,\FFk)_{\rhobar,E} \cong \widetilde{X}_\q(\up,\OO_\pf)_{\rhobar}\otimes_{\OO_\pf} W_{k,\pf}.$$
\end{lemma}

\subsection{Mazur's principle}We can now proceed as in Theorem 6.2.4 of \cite{Emlg}, following Jarvis's approach to Mazur's principle over totally real fields \cite{JarMazPrin}. We retain a fixed absolutely irreducible mod $p$ Galois representation $\rhobar$, coming from a maximal ideal of $\T(\up)$. We have the following proposition
\begin{proposition}\label{unramred}
The injection $$\widetilde{H}^1_{\red}(\up,\OO_\pf)_{\rhobar} \hookrightarrow \widetilde{H}^1_D(\up,\OO_\pf)_{\rhobar}$$ induces an isomorphism
$$\widetilde{H}^1_{\red}(\up,\OO_\pf)_{\rhobar} \cong (\widetilde{H}^1_D(\up,\OO_\pf)_{\rhobar})^{I_\q}.$$
\end{proposition}
\begin{proof}
This follows from Proposition 4 and (2.2) in \cite{Raj}.
\end{proof}
There is a Hecke operator $T_\q$ acting on $H^1_D(\up,\OO_\pf)$, which extends by continuity to give an action of $T_\q$ on $\widetilde{H}^1_D(\up,\OO_\pf)$.

\begin{lemma}\label{frobaction}
If we view $\widetilde{H}^1_\mathrm{red}(\up,\OO_\pf)_{\rhobar}$ as a submodule of $\widetilde{H}^1_D(\up,\OO_\pf)_{\rhobar}$ via the first exact sequence of Proposition \ref{prop:es}, then it is stable under the action of $T_\q$. The induced action of $(\mathbf{N}\q)^{-1}T_\q$ on $\widetilde{X}_\q(\up,\OO_\pf)_{\rhobar}(-1)$ is equal to the action of the arithmetic Frobenius element of $\Gal(\overline{k}_\q/k_\q)$. Viewing $\widetilde{\check{X}}_\q(\up,\OO_\pf)_{\rhobar}$ as a submodule of $\widetilde{H}^1_\mathrm{red}(\up,\OO_\pf)_{\rhobar}$ via the second exact sequence of Proposition \ref{prop:es}, it is also $T_\q$-stable, with the action of $T_\q$ equal to the action of arithmetic Frobenius.
\end{lemma}
\begin{proof}
It is sufficient to check the statement of the lemma on the dense subspaces ${H}^1_\mathrm{red}(\up,\OO_\pf)_{\rhobar}$, $\check{X}_\q (\up,\OO_\pf)_{\rhobar}$ and $X_\q(\up,\OO_\pf)_{\rhobar}(-1)$. The lemma now follows from local--global compatibility in the case of twists of the Steinberg representation by unramified characters, as proven in Section 6 of \cite{Car}.
\end{proof}

\begin{lemma}\label{oldred}
There is a natural isomorphism $$\widetilde{H}^1_{\red}(\up,r_*r^*\OO_\pf)_{\rhobar} \cong \widetilde{H}^1_D(U^p,\OO_\pf)_{\rhobar}^{\oplus 2}.$$
\end{lemma}
\begin{proof}
This follows from Lemma 16.1 in \cite{JarMazPrin}.
\end{proof}
\begin{lemma}\label{monodromy}
The monodromy pairing maps $$\lambda: X_\q(U_p\up,\OO_\pf) \rightarrow \check{X}_\q(U_p\up,\OO_\pf)$$ induce a $G_\q$ and $\T(\up)_{\rhobar}$-equivariant isomorphism $$\widetilde{X}_\q(\up,\OO_\pf)_{\rhobar} \cong \widetilde{\check{X}}_\q(\up,\OO_\pf)_{\rhobar}.$$
\end{lemma}
\begin{proof}
This follows from Proposition 5 in \cite{Raj}.
\end{proof}
%\begin{remark}
%Note that the isomorphism in the above lemma is not $T_\q$ equivariant.
%\end{remark}
Before giving an analogue of Mazur's principle, we will make explicit the connection between the modules $\widetilde{X}_\q(\up,\OO_\pf)_{\rhobar}$ and a space of newforms. There are natural level raising maps $i:\widetilde{H}^1_D(U^p,\OO_\pf)_{\rhobar}^{\oplus 2} \rightarrow \widetilde{H}^1_D(\up,\OO_\pf)_{\rhobar}$, and their adjoints $i^\dagger: \widetilde{H}^1_D(\up,\OO_\pf)_{\rhobar} \rightarrow \widetilde{H}^1_D(U^p,\OO_\pf)_{\rhobar}^{\oplus 2}$, defined by taking the limit of maps at finite level defined as in the paragraph after Theorem \ref{RRESclass}.
\begin{definition}\label{def:new}
The space of $\q$-newforms $$\widetilde{H}^1_D(\up,\OO_\pf)_{\rhobar}^{\q\mhyphen\mathrm{new}}$$ is defined to be the kernel of the map $$i^\dagger: \widetilde{H}^1_D(\up,\OO_\pf)_{\rhobar} \rightarrow \widetilde{H}^1_D(U^p,\OO_\pf)_{\rhobar}^{\oplus 2}.$$
\end{definition}
Composing the injection $\widetilde{H}^1_\mathrm{red}(\up,\OO_\pf)_{\rhobar} \hookrightarrow \widetilde{H}^1_D(\up,\OO_\pf)_{\rhobar}$ with $i^\dagger$ gives a map, which we also denote by $i^\dagger$, from $\widetilde{H}^1_\mathrm{red}(\up,\OO_\pf)_{\rhobar}$ to $\widetilde{H}^1_D(U^p,\OO_\pf)_{\rhobar}^{\oplus 2}.$ Recall that we can also define a map $\omega$ between these spaces, by composing the natural map $\widetilde{H}^1_\mathrm{red}(\up,\OO_\pf)_{\rhobar} \rightarrow \widetilde{H}^1_\mathrm{red}(\up,r_*r^*\OO_\pf)_{\rhobar}$ with the isomorphism of Lemma \ref{oldred}. However these maps are not equal, as remarked in \cite[pg. 447]{Ribet100}. Fortunately, they have the same kernel, as we now show (essentially by following the proof of \cite[Theorem 3.11]{Ribet100}).
\begin{proposition}\label{prop:fix}
The kernel of the map $$i^\dagger: \widetilde{H}^1_\mathrm{red}(\up,\OO_\pf)_{\rhobar}\rightarrow\widetilde{H}^1_D(U^p,\OO_\pf)_{\rhobar}^{\oplus 2}$$ is equal to the image of $\widetilde{\check{X}}_\q(\up,\OO_\pf)_{\rhobar}$ in $\widetilde{H}^1_\mathrm{red}(\up,\OO_\pf)_{\rhobar}$.
\end{proposition}
\begin{proof}
Since all our spaces are $\OO_\pf$-torsion free, it suffices to show that the kernel of the map $$i^\dagger: H^1(\Mi(U_p\up)_{\overline{s}}, \OO_\pf)_{\rhobar} \rightarrow H^1(\Mi(U_pU^p)_{\overline{s}}, \OO_\pf)_{\rhobar}^{\oplus 2}$$ is equal to the image of $\check{X}_\q(U_p\up,\OO_\pf)_{\rhobar}$ in $H^1(\Mi(U_p\up)_{\overline{s}}, \OO_\pf)_{\rhobar}$ for every $U_p$. In fact this just needs to be checked after tensoring everything with $E$. For brevity we write $U$ for $U_pU^p$ and $U(\q)$ for $U_p\up$. 
We complete the proof by relating our constructions to the Jacobians of Shimura curves. The degeneracy maps $\rho_1,\rho_{\eta_\q}$ induce by Picard functoriality a map $\rho:\Jac(M(U)/F) \times \Jac(M(U)/F) \rightarrow \Jac(M(U(\q))/F)$ with finite kernel, giving the map $i$ on cohomology. The dual map $\rho^\dagger$ (induced by Albanese functoriality from the degeneracy maps) gives the map $i^\dagger$ on cohomology. Let $J^0$ denote the connected component of $0$ in the special fibre at $\q$ of the N\'{e}ron model of $\Jac(M(U(\q))/F)$. The map $\omega$ arises from a map $\Omega: J^0 \rightarrow \Jac(\Mi(U)_{\overline{s}}) \times \Jac(\Mi(U)_{\overline{s}})$, which identifies $\Jac(\Mi(U)_{\overline{s}}) \times \Jac(\Mi(U)_{\overline{s}})$ as the maximal Abelian variety quotient of $J^0$ (which is a semi-Abelian variety). The composition of $\Omega$ with the map induced by $\rho$ on the special fibre gives an auto-isogeny $A$ of $\Jac(\Mi(U)_{\overline{s}}) \times \Jac(\Mi(U)_{\overline{s}})$. The quasi-auto-isogeny $(\deg{\rho})^{-1}A$ gives an automorphism of $H^1(\Mi(U_pU^p)_{\overline{s}}, E)_{\rhobar}^{\oplus 2}$. Denote the inverse of this automorphism by $\gamma$. Then the construction of the automorphism $\gamma$ implies that there is a commutative diagram
$$\minCDarrowwidth10pt\begin{CD}H^1(\Mi(U_p\up)_{\overline{s}}, E)_{\rhobar} @>{\omega}>>H^1(\Mi(U_pU^p)_{\overline{s}}, E)_{\rhobar}^{\oplus 2}@.\\
@| @V{\gamma}VV \\
H^1(\Mi(U_p\up)_{\overline{s}}, E)_{\rhobar} @>{i^\dagger}>>H^1(\Mi(U_pU^p)_{\overline{s}}, E)_{\rhobar}^{\oplus 2}@.,\end{CD}$$
which identifies the kernel of $i^\dagger$ with the kernel of $\omega$.
\end{proof}

\begin{proposition}\label{prop:Xnew}
A system of Hecke eigenvalues $\lambda: \T(\up) \rightarrow E'$ occurs in $\widetilde{H}^1_D(\up,\OO_\pf)_{\rhobar}^{\q\mhyphen\mathrm{new}}$ if and only if it occurs in $\widetilde{X}_\q(\up,\OO_\pf)_{\rhobar}$.
\end{proposition}
\begin{proof}
Proposition \ref{prop:es} and Lemma \ref{oldred}, together with Definition \ref{def:new} and Proposition \ref{prop:fix} imply that we have the following commutative diagram of $\varpi$-adically admissible $G(\Q_p)$ representations over $\OO_\pf$:
$$\minCDarrowwidth10pt\begin{CD}0@>>>\widetilde{\check{X}}_\q(\up,\OO_\pf)_{\rhobar}@>>>\widetilde{H}^1_{\red}(\up,\OO_\pf)_{\rhobar} @>{\omega}>>\widetilde{H}^1_D(U^p,\OO_\pf)_{\rhobar}^{\oplus 2}@>>>0\\
@. @V{\alpha}VV @V{\beta}VV @V{\gamma}VV\\
0@>>>\widetilde{H}^1_D(\up,\OO_\pf)_{\rhobar}^{\q\mhyphen\mathrm{new}}@>>>\widetilde{H}^1_D(\up,\OO_\pf)_{\rhobar} @>{i^\dagger}>>\widetilde{H}^1_D(U^p,\OO_\pf)_{\rhobar}^{\oplus 2},\end{CD}$$
where the two rows are exact and the maps $\alpha$, $\beta$ and $\gamma$ are injections. Note that both the existence of the map $\gamma$ and its injectivity follow from Proposition \ref{prop:fix}. Moreover $$\coker(\beta)=\widetilde{X}_\q(\up,\OO_\pf)_{\rhobar}(-1) \cong \widetilde{\check{X}}_\q(\up,\OO_\pf)_{\rhobar}(-1),$$ where the second isomorphism comes from the monodromy pairing (see Lemma \ref{monodromy}). Applying the snake lemma we see that $\coker(\alpha)\hookrightarrow\coker(\beta)$, so (applying Lemma \ref{monodromy} once more to the source of the map $\alpha$) we have an exact sequence
$$\minCDarrowwidth10pt\begin{CD}0@>>>\widetilde{X}_\q(\up,\OO_\pf)_{\rhobar} @>\alpha>>\widetilde{H}^1_D(\up,\OO_\pf)_{\rhobar}^{\q\mhyphen\mathrm{new}} @>>>\widetilde{X}_\q(\up,\OO_\pf)_{\rhobar}(-1).\end{CD}$$
From this exact sequence it is easy to deduce that the systems of Hecke eigenvalues occurring in $\widetilde{H}^1_D(\up,\OO_\pf)_{\rhobar}^{\q\mhyphen\mathrm{new}}$ and $\widetilde{X}_\q(\up,\OO_\pf)_{\rhobar}$ are the same.
\end{proof}
\begin{theorem}\label{mazprin}
Let $\rho: \Gal(\overline{F}/F) \rightarrow \GL_2(\overline{\Q}_p)$ be a Galois representation, with irreducible reduction $\overline{\rho}: \Gal(\overline{F}/F) \rightarrow \GL_2(\overline{\F}_p)$. Suppose the following two conditions are verified: 
\begin{enumerate}\item $\rho$ is unramified at the prime $\q$
\item $\rho(\mathrm{Frob}_\q)$ is not a scalar
\item There is a system of Hecke eigenvalues $\lambda:\T(U^p)_{\rhobar} \rightarrow E'$ attached to $\rho$ (i.e. $\rho \cong \rho(\ker(\lambda))$), with $\lambda$ occurring in $\widetilde{H}^1_D(\up,\OO_\pf).$
\end{enumerate} Then $\lambda$ occurs in $\widetilde{H}^1_D(U^p,\OO_\pf)$.
\end{theorem}
\begin{proof}
We take $\rho$ as in the statement of the theorem. The third assumption, combined with Proposition \ref{galfac} implies that there is a prime ideal $\mathscr{P}=\ker(\lambda)$  of $\T(\up)$, and an embedding $$\rho \hookrightarrow \widetilde{H}^1_D(\up,\OO_\pf)[\mathscr{P}].$$ Our aim is to show that $\widetilde{H}^1_D(U^p,\OO_\pf)[\mathscr{P}]\ne 0$. 

Let $\T(\up)^*$ denote the ring extension of $\T(\up)$ obtained by adjoining the operator $T_\q$. Since $\rho$ is irreducible there is a prime $\mathscr{P}^*$ of $\T(\up)^*$ such that there is an embedding $$\rho \hookrightarrow \widetilde{H}^1_D(\up,\OO_\pf)[\mathscr{P}^*].$$

Recall that by Propositions \ref{prop:es} and \ref{unramred} and Lemma \ref{oldred} there is a $G_\q$ and $\T(\up)$-equivariant short exact sequence $$\minCDarrowwidth10pt\begin{CD}0@>>>\widetilde{\check{X}}_\q(\up,\OO_\pf)_{\rhobar}@>>>\widetilde{H}^1_D(\up,\OO_\pf)_{\rhobar}^{I_\q} @>{\omega}>>\widetilde{H}^1_D(U^p,\OO_\pf)_{\rhobar}^{\oplus 2}@>>>0.\end{CD}$$ Suppose for a contradiction that $\widetilde{H}^1_D(U^p,\OO_\pf)_{\rhobar}[\mathscr{P}]=0$

Let $V$ denote the $G_\q$-representation obtained from $\rho$ by restriction. Since $\rho$ is unramified at $\q$, there is an embedding $V \hookrightarrow \widetilde{H}^1_D(\up,\OO_\pf)^{I_\q}_{\rhobar}[\mathscr{P}^*]$. The above short exact sequence then implies that the embedding $$\widetilde{\check{X}}_\q(\up,\OO_\pf)_{\rhobar}[\mathscr{P}]\hookrightarrow \widetilde{H}^1_D(\up,\OO_\pf)_{\rhobar}^{I_\q}[\mathscr{P}]$$ is an isomorphism, and hence the embedding $$\widetilde{\check{X}}_\q(\up,\OO_\pf)_{\rhobar}[\mathscr{P}^*]\hookrightarrow \widetilde{H}^1_D(\up,\OO_\pf)_{\rhobar}^{I_\q}[\mathscr{P}^*]$$ is also an isomorphism.

This implies that there is an embedding $V \hookrightarrow \widetilde{\check{X}}_\q(\up,\OO_\pf)_{\rhobar}[\mathscr{P}^*]$, so by Lemma \ref{frobaction} $\rho(\mathrm{Frob}_\q)$ is a scalar $\alpha$ where $\alpha$ is the $T_\q$ eigenvalue coming from $\mathscr{P}^*$, contradicting our second assumption.
\end{proof}
\subsection{Mazur's principle and the Jacquet functor}\label{MazJac}
	Recall that given an admissible Banach $G(\Q_p)$-representation $V$ over $\ep$ (for example, the extension to $\ep$ of a $\varpi$-adically admissible $G(\Q_p)$-representation over $\OO_\pf$), Schneider and Tei\-telbaum \cite[\S 7]{STei} define a functor by passing to the locally $\Q_p$-analytic vectors $V^{an}$. This functor is exact \cite[Theorem 7.1]{STei}. We let $B$ denote the Borel subgroup of $G(\Q_p)$ consisting of upper triangular matrices, and let $T$ denote the maximal torus contained in $B$. We then have a locally analytic Jacquet functor $J_B$ as defined in \cite{MR2292633}, which can be applied to the space of locally analytic vectors $V^{an}$ of an admissible continuous Banach $E[G(\Q_p)]$-representation $V$. Since the locally analytic Jacquet functor is left exact, the composition of passing to locally analytic vectors and then applying the Jacquet functor is left exact. For brevity, we will also denote this composition by $J_B$. The following lemma follows from Proposition \ref{galfac}, but we give a more elementary separate proof.
\begin{lemma}\label{galembed}
Let $\lambda:\T(\up)\rightarrow E$ be a system of Hecke eigenvalues such that $\lambda$ occurs in $J_B(\widetilde{H}^1_D(\up,E))$. Suppose the attached Galois representation $\rho$ is absolutely irreducible.Then there is a non-zero $\Gal(\overline{F}/F)$-equivariant map (necessarily an embedding, since $\rho$ is irreducible) $$\rho \hookrightarrow J_B(\widetilde{H}^1_D(\up,E)).$$
\end{lemma}
\begin{proof}
To abbreviate notation we let $M$ denote $J_B(\widetilde{H}^1_D(\up,E))$. Since $M$ is an essentially admissible $T$ representation, the system of Hecke eigenvalues $\lambda$ occurs in the $\chi$-isotypic subspace $M^\chi$ for some continuous character $\chi$ of $T$. So (again by essential admissibility) $M^{\chi,\lambda}$ is a non-zero finite dimensional $\Gal(\overline{F}/F)$-representation over $E$, and therefore by the Eichler-Shimura relations (as in section 10.3 of \cite{carmauv}) and the irreducibility of $\rho$ the desired embedding exists.
\end{proof}
\begin{lemma}\label{commute}
The natural embedding $$J_B(\widetilde{H}^1_D(\up,E)^{I_\q}_{\rhobar}) \rightarrow J_B(\widetilde{H}^1_D(\up,E)_{\rhobar})$$ has image equal to $J_B(\widetilde{H}^1_D(\up,E)_{\rhobar})^{I_\q}$.
\end{lemma}
\begin{proof}
As remarked above, the composition of the Jacquet functor with passing to locally analytic vectors is a left exact functor. On the other hand, the action of $I_\q$ on $\widetilde{H}^1_D(\up,E)$ factors through the tame inertia group, since this is so on each of the finite $\OO_\pf/\varpi^s$-modules $$H^1(M(U_p \up)_{\overline{F}},\OO_\pf/\varpi^s).$$ Hence, we can think of the $I_\q$-invariants of $\widetilde{H}^1_D(\up,E)_{\rhobar}$ as the kernel of $\tau-1 \in E[G_\q]$ acting on the representation, where $\tau$ is a lift of a generator of tame inertia. Since left exact functors preserve kernels, we obtain the statement of the lemma.
\end{proof}
\begin{theorem}\label{mazprinjacquet}
Suppose that the system of Hecke eigenvalues $\lambda:\T(\up)\rightarrow E$ occurs in $J_B(\widetilde{H}^1_D(\up,E))$, that the attached Galois representation $$\rho : \Gal(\overline{F}/F) \rightarrow \GL_2(E)$$ is unramified at $\q$, with $\rho(Frob_\q)$ not a scalar, and that the reduction $$\overline{\rho}: \Gal(\overline{F}/F) \rightarrow \GL_2(\overline{\F}_p)$$ is irreducible. Then $\lambda$ occurs in $J_B(\widetilde{H}^1_D(U^p,E))$.
\end{theorem}
\begin{proof}
We take $\rho$ as in the statement of the theorem. Let $\m$ be the maximal ideal of $\T(\up)$ attached to $\overline{\rho}$, and let $\mathscr{P}$ be the prime ideal of $\T(\up)$ attached to $\rho$, so applying Lemma \ref{galembed} gives an embedding $$\rho \hookrightarrow J_B(\widetilde{H}^1_D(\up,E)_{\rhobar})[\mathscr{P}].$$ Our aim is to show that $J_B(\widetilde{H}^1_D(U^p,E)_{\rhobar})[\mathscr{P}]\ne 0$. Let $\T(\up)^*$ denote the finite ring extension of $\T(\up)$ obtained by adjoining the operator $T_\q$. Since $\rho$ is irreducible there is a prime $\mathscr{P}^*$ of $\T(\up)^*$ such that there is an embedding $$\rho \hookrightarrow J_B(\widetilde{H}^1_D(\up,E)_{\rhobar})[\mathscr{P}^*].$$
By Propositions \ref{prop:es} and \ref{unramred} and Lemma \ref{oldred} there is a short exact sequence
$$\minCDarrowwidth10pt\begin{CD}0@>>>\widetilde{\check{X}}_\q(\up,E)_{\rhobar}@>>>\widetilde{H}^1_D(\up,E)_{\rhobar}^{I_\q} @>>>\widetilde{H}^1_D(U^p,E)_{\rhobar}^{\oplus 2}@>>>0.\end{CD}$$
Taking locally analytic vectors and applying the Jacquet functor, we get an exact sequence
\small$$\minCDarrowwidth10pt\begin{CD}0@>>>J_B(\widetilde{\check{X}}_\q(\up,E)_{\rhobar})@>>>J_B((\widetilde{H}^1_D(\up,E)^{I_\q}_{\rhobar})) @>{\omega}>>J_B(\widetilde{H}^1_D(U^p,E)_{\rhobar})^{\oplus 2}.\end{CD}$$\normalsize

Let $V$ denote the $G_\q$-representation obtained from $\rho$ by restriction. Since $\rho$ is unramified at $\q$, by Lemma \ref{commute} there is an embedding $V\hookrightarrow J_B(\widetilde{H}^1_D(\up,E)^{I_\q}_{\rhobar})[\mathscr{P}^*]$. Suppose for a contradiction that $J_B(\widetilde{H}^1_D(U^p,E)_{\rhobar})[\mathscr{P}]=0$. The above exact sequence then implies that the embedding $$J_B(\widetilde{\check{X}}_\q(\up,E)_{\rhobar})[\mathscr{P}]\hookrightarrow J_B(\widetilde{H}^1_D(\up,E)^{I_\q}_{\rhobar}) [\mathscr{P}]$$ is an isomorphism, and hence the embedding $$J_B(\widetilde{\check{X}}_\q(\up,E)_{\rhobar})[\mathscr{P}^*]\hookrightarrow J_B(\widetilde{H}^1_D(\up,E)^{I_\q}_{\rhobar}) [\mathscr{P}^*]$$ is also an isomorphism.

This implies that there is an embedding $V \hookrightarrow J_B(\widetilde{\check{X}}_\q(\up,E)_{\rhobar})[\mathscr{P}^*]$, so by Lemma \ref{frobaction} $\rho(\mathrm{Frob}_\q)$ is a scalar $\alpha$ corresponding to the $T_\q$ eigenvalue in $\mathscr{P}^*$. We then obtain a contradiction exactly as in the proof of Theorem \ref{mazprin}.
\end{proof}

\subsection{A $p$-adic Jacquet--Langlands correspondence}\label{pJL}
We now wish to relate the arithmetic of $p$-adic automorphic forms for $G$ and $G'$. The definitions and results of subsection \ref{subsec:ccoho} apply to $G'$. We denote the associated completed cohomology spaces of tame level $V^p$ by $$\widetilde{H}^1_{D'}(V^p,\FF_k),$$ and its Hecke algebra by $\T'(V^p)$. Fix a non-Eisenstein maximal ideal $\m'$ of $\T'(V^p)$, with attached mod $p$ Galois representation $\rhobar$. The integral models $\Mi'$ and the vanishing cycle formalism also allows us to define some more topological $\OO_\pf$-modules:

\begin{definition}\label{Ycompletedstuff}
We define $\OO_\pf$-modules 
$$Y_\q(V^p,\FFk)_{\rhobar}=\varinjlim_{V_p}Y_\q(V_pV^p,\FFk)_{\rhobar},$$ and $$\check{Y}_\q(V^p,\FFk)_{\rhobar}=\varinjlim_{V_p}\check{Y}_\q(V_pV^p,\FFk)_{\rhobar}.$$

We define $\OO_\pf$-modules 
$$\widetilde{Y}_\q(V^p,\FFk)_{\rhobar} = \varprojlim_s Y(V^p,\FFk)_{\rhobar}/\varpi^sY(V^p,\FFk)_{\rhobar}$$ and $$\widetilde{\check{Y}}_\q(V^p,\FFk)_{\rhobar} = \varprojlim_s \check{Y}_\q(V^p,\FFk)_{\rhobar}/\varpi^s\check{Y}_\q(V^p,\FFk)_{\rhobar}.$$
\end{definition}

As before, these modules acquire $G(\Q_p)$ actions after extending scalars to $E$, and if $\FFk = \OO_\pf$ there is even a $G(\Q_p)$ action on the $\OO_\pf$-module before extending scalars.

In exactly the same way as we obtained Proposition \ref{prop:es} we get:
\begin{proposition}\label{prop:Yses}
We have a short exact sequence equivariant with respect to $\T'(V^p)_{\rhobar}$ and $G_\q$ actions:
$$\minCDarrowwidth10pt\begin{CD}0@>>>\widetilde{\check{Y}}_\q(V^p,\FFk)_{\rhobar}@>>>\widetilde{H}^1_{D'}(V^p,\FFk)_{\rhobar}@>>>\widetilde{Y}_\q(V^p,\FFk)_{\rhobar}(-1)@>>>0.\end{CD}$$
If $\FFk = \OO_\pf$ this is a short exact sequence of $\varpi$-adically admissible $G(\Q_p)$-representations over $\OO_\pf$.
\end{proposition}
We also have
\begin{lemma}\label{Ymonodromy}
The monodromy pairing maps $\lambda: Y_\q(V_pV^p,\OO_\pf) \rightarrow \check{Y}_\q(V_pV^p,\OO_\pf)$ induce a $G_\q$ and $\T'(V^p)_{\rhobar}$-equivariant isomorphism $$\widetilde{Y}_\q(V^p,E)_{\rhobar} \cong \widetilde{\check{Y}}_\q(V^p,E)_{\rhobar}.$$
\end{lemma}
\begin{proof}
The proof is similar to the first part of the proof of Proposition \ref{galfac}. We verify that the map $\lambda_E: \widetilde{Y}_\q(V^p,E)_{\rhobar} \rightarrow \widetilde{\check{Y}}_\q(V^p,E)_{\rhobar}$ induced by the monodromy pairing at finite levels is surjective (its injectivity follows from the injectivity at finite levels). Let $\mathscr{P} \in \Spec(\T'(V^p)_{\rhobar}[\frac{1}{p}])$ be a prime ideal of the Hecke algebra arising from a Hecke eigenform in $$H^1(M'(V^pV_p),\OO_\pf)_{\rhobar}\otimes_{\OO_\pf}k(\mathscr{P}).$$ 
Then the image of the map $$\widetilde{Y}_\q(V^p,E)_{\rhobar}[\mathscr{P}] \rightarrow \widetilde{\check{Y}}_\q(V^p,E)_{\rhobar}[\mathscr{P}]$$ induced by $\lambda$ contains $\check{Y}_\q(V_pV^p,E)_{\rhobar}[\mathscr{P}]$ (the map $\lambda: Y_\q(V_pV^p,\OO_\pf) \rightarrow \check{Y}_\q(V_pV^p,\OO_\pf)$ becomes an isomorphism after extending scalars to the field $E$, by the weight-monodromy conjecture for curves). Hence the image of the map $$\lambda_E: \widetilde{Y}_\q(V^p,E)_{\rhobar}\rightarrow \widetilde{\check{Y}}_\q(V^p,E)_{\rhobar}$$ is dense, since it contains the dense subspace $\check{Y}_\q(V^p,E)_{\rhobar}$. Now since $\widetilde{Y}_\q(V^p,\OO_\pf)_{\rhobar}$ is a $\varpi$-adically admissible $G(\Q_p)$-representation over $\OO_\pf$, \cite[Proposition 3.1.3]{Emlg} implies that the image of $\lambda_E$ is closed, hence $\lambda$ is surjective. \end{proof}

Fix $U^p \subset G(\A_f^{(p)})$ such that $U^p$ is unramified at the places $\qa,\qb$, and matches with $V^p$ at all other places. The Galois representation $\rhobar$ arises from a maximal ideal $\m'_0$ of $\T'_{(2,...,2)}(V^pV_p)$ for $V_p$ a small enough compact open subgroup of $G'(\Q_p)$, and this pulls back by the map $\T_{(2,...,2)}(U^p(\qa\qb)U_p) \rightarrow \T'_{(2,...,2)}(V^pV_p)$ to a maximal ideal $\m_0$ of $\T_{(2,...,2)}(U^p(\qa\qb)U_p)$ for $U_p$ the compact open subgroup of $G(\Q_p)$ which is identified with $V_p$ under the isomorphism $G(\Q_p) \cong G'(\Q_p)$. Hence there is a map $\T(U^p(\qa\qb))_{\rhobar} \rightarrow \T'(V^p)_{\rhobar}$.

\begin{theorem}\label{RRES}
We have the following short exact sequences of $\varpi$-adically admissible $G(\Q_p)$-representations over $\OO_\pf$, equivariant with respect to $\T(U^p(\qa\qb))_{\rhobar}$ and $G(\Q_p)\cong G'(\Q_p)$ actions:
$$\minCDarrowwidth10pt\begin{CD}0@>>>\widetilde{Y}_\qa(V^p,\OO_\pf)_{\rhobar} @>>>\widetilde{X}_\qb(U^p(\qa\qb),\OO_\pf)_{\rhobar}@>{i^\dagger}>>\widetilde{X}_\qb(U^p(\qb),\OO_\pf)^{\oplus 2}_{\rhobar}@>>>0,\end{CD}$$
$$\minCDarrowwidth10pt\begin{CD}0@>>>\widetilde{\check{X}}_\qb(U^p(\qb),\OO_\pf)^{\oplus 2}_{\rhobar}@>>>\widetilde{\check{X}}_\qb(U^p(\qa\qb),\OO_\pf)_{\rhobar} @>>>\widetilde{\check{Y}}_\qa(V^p,\OO_\pf)_{\rhobar}@>>>0.\end{CD}$$
\end{theorem}
\begin{proof}
The proof is as for Proposition \ref{prop:es}. We just need to check that the short exact sequences of Theorem \ref{RRESclass} are compatible as $U_p$ and $V_p$ vary, but this is clear from the proof of \cite[Theorem 3]{Raj}, since the descriptions of the dual graph of the special fibres of $\mathbb{M}_\qb(U^p(\qa\qb)U_p)$ and $\mathbb{M}'_\qa(V^pV_p)$ are compatible as $U_p$ and $V_p$ vary.
\end{proof}
We regard Theorem \ref{RRES} (especially the first exact sequence therein) as a geometric realisation of a $p$-adic Jacquet--Langlands correspondence. The following propositions explain why this is reasonable. We let $\widetilde{X}_\qb(U^p(\qa\qb),\OO_\pf)_{\rhobar}^{\qa\mhyphen\mathrm{new}}$ denote the kernel of the map $$i^\dagger: \widetilde{X}_\qb(U^p(\qa\qb),\OO_\pf)_{\rhobar}\rightarrow \widetilde{X}_\qb(U^p(\qb),\OO_\pf)^{\oplus 2}_{\rhobar}.$$ We also denote by $\widetilde{H}^1_D(U^p(\qa\qb),\OO_\pf)_{\rhobar}^{\qa\qb\mhyphen\mathrm{new}}$ the intersection $$\widetilde{H}^1_D(U^p(\qa\qb),\OO_\pf)_{\rhobar}^{\qa\mhyphen\mathrm{new}}\cap\widetilde{H}^1_D(U^p(\qa\qb),\OO_\pf)_{\rhobar}^{\qb\mhyphen\mathrm{new}}.$$
\begin{proposition}\label{prop:Xdoublenew}
A system of Hecke eigenvalues $\lambda: \T(U^p(\qa\qb)) \rightarrow E'$ occurs in $\widetilde{H}^1_D(U^p(\qa\qb),\OO_\pf)_{\rhobar}^{\qa\qb\mhyphen\mathrm{new}}$ if and only if it occurs in $\widetilde{X}_\qb(U^p(\qa\qb),\OO_\pf)_{\rhobar}^{\qa\mhyphen\mathrm{new}}$.
\end{proposition}
\begin{proof}
We have the following commutative diagram:
$$\minCDarrowwidth20pt\begin{CD}0@. 0@. \\
@VVV @VVV \\
\widetilde{X}_\qb(U^p(\qa\qb),\OO_\pf)_{\rhobar}^{\qa\mhyphen\mathrm{new}} @>{\alpha}>>\widetilde{H}^1_D(U^p(\qa\qb),\OO_\pf)_{\rhobar}^{\qa\qb\mhyphen\mathrm{new}} @.\\
@VVV @VVV \\
\widetilde{X}_\qb(U^p(\qa\qb),\OO_\pf)_{\rhobar} @>{\beta}>> \widetilde{H}^1_D(U^p(\qa\qb),\OO_\pf)^{\qb\mhyphen\mathrm{new}}_{\rhobar} @.\\
@VVV @VVV \\
\widetilde{X}_\qb(U^p(\qb),\OO_\pf)^{\oplus 2}_{\rhobar} @>{\gamma}>> (\widetilde{H}^1_D(U^p(\qb),\OO_\pf)_{\rhobar}^{\qb\mhyphen\mathrm{new}})^{\oplus 2} @.\\
@VVV @. \\
0@. @.\end{CD}$$
where the two columns are exact, and the maps $\alpha$, $\beta$ and $\gamma$ are injections. Note that we are applying Lemma \ref{monodromy}. By the snake lemma we have an exact sequence
$$\minCDarrowwidth15pt\begin{CD}0@>>> \coker(\alpha) @>>> \coker(\beta) @>>> \coker(\gamma),\end{CD}$$
and the proof of Proposition \ref{prop:Xnew} implies that there are injections $$\begin{CD}@.\coker(\beta) \hookrightarrow \widetilde{X}_\qb(U^p(\qa\qb),\OO_\pf)_{\rhobar}(-1),\\@.\coker(\gamma) \hookrightarrow \widetilde{X}_\qb(U^p(\qb),\OO_\pf)^{\oplus 2}_{\rhobar}(-1).\end{CD}$$ This shows that there is a Hecke-equivariant injection $$\coker(\alpha)\hookrightarrow\widetilde{X}_\qb(U^p(\qa\qb),\OO_\pf)_{\rhobar}^{\qa\mhyphen\mathrm{new}}(-1).$$ Therefore, we have an exact sequence $$\minCDarrowwidth15pt\begin{CD}0@>>>\widetilde{X}_\qb(U^p(\qa\qb),\OO_\pf)_{\rhobar}^{\qa\mhyphen\mathrm{new}}@>\alpha>>\widetilde{H}^1_D(U^p(\qa\qb),\OO_\pf)_{\rhobar}^{\qa\qb\mhyphen\mathrm{new}} \\@. @>>>\widetilde{X}_\qb(U^p(\qa\qb),\OO_\pf)_{\rhobar}^{\qa\mhyphen\mathrm{new}}(-1).\end{CD}$$
From this last exact sequence we deduce that the systems of Hecke eigenvalues occurring in $\widetilde{H}^1_D(U^p(\qa\qb),\OO_\pf)_{\rhobar}^{\qa\qb\mhyphen\mathrm{new}}$ and $\widetilde{X}_\qb(U^p(\qa\qb),\OO_\pf)_{\rhobar}^{\qa\mhyphen\mathrm{new}}$ are the same.
\end{proof}
\begin{proposition}\label{YHsame}
A system of Hecke eigenvalues $\lambda: \T'(V^p) \rightarrow E'$ occurs in \\$\widetilde{H}_{D'}^1(V^p,\OO_\pf)_{\rhobar}$ if and only if it occurs in $\widetilde{Y}_\qa(V^p,\OO_\pf)_{\rhobar}$.
\end{proposition}
\begin{proof}
This follows from Lemma \ref{Ymonodromy} and the short exact sequence of Proposition \ref{prop:Yses}.
\end{proof}
Combined with Theorem \ref{RRES} this has as a consequence
\begin{corollary}
The first exact sequence of Theorem \ref{RRES} induces a bijection between the systems of eigenvalues for $\T'(V^p)$ appearing in $\widetilde{H}_{D'}^1(V^p,\OO_\pf)_{\rhobar}$ and the systems of eigenvalues for $\T(U^p(\qa\qb))$ appearing in $\widetilde{H}^1_D(U^p(\qa\qb),\OO_\pf)_{\rhobar}^{\qa\qb\mhyphen\mathrm{new}}$.
\end{corollary}

\section{Eigenvarieties and an overconvergent Jacquet--Langlands correspondence}
In this section we will apply Emerton's construction of eigenvarieties to deduce an overconvergent Jacquet--Langlands correspondence from Theorem \ref{RRES}. We will again be using the functor of passing to locally analytic vectors, as discussed in (\ref{MazJac}) above.
\subsection{Locally algebraic vectors}\label{localg}
Fix a maximal ideal $\m$ of $\T(\up)$ which is non-Eisenstein, with associated mod $p$ Galois representation $\rhobar$. Recall that given an admissible Banach $G(\Q_p)$-representation $V$ over $\ep$, the $G(\Q_p)$-action on $V^{an}$ differentiates to give an action of the Lie algebra $\g$ of $G(\Q_p)$. 
\begin{proposition}
There are natural isomorphisms (for $\FF = \FFk$ or $\rFFk$) \begin{enumerate}\item $H^0(\g,\widetilde{H}^1_D(\up,\FFk)^{an}_{\rhobar,\ep})\cong H^1_D(\up,\FFk)_{\rhobar,\ep},$
\item $H^1(\g,\widetilde{H}^1_D(\up,\FFk)^{an}_{\rhobar,\ep})\cong 0,$
\item $H^0(\g,\widetilde{H}^1_\red(\up,\FF)^{an}_{\rhobar,\ep})\cong H^1_\red(\up,\FF)_{\rhobar,\ep},$
\item $H^1(\g,\widetilde{H}^1_\red(\up,\FF)^{an}_{\rhobar,\ep})\cong 0,$
\item $H^0(\g,\widetilde{X}_\q(\up,\FFk)^{an}_{\rhobar,\ep})\cong X_\q (\up,\FFk)_{\rhobar,\ep},$
\item $H^0(\g,\widetilde{\check{X}}_\q(\up,\FFk)^{an}_{\rhobar,\ep})\cong \check{X}_\q(\up,\FFk)_{\rhobar,\ep},$\end{enumerate}
\end{proposition}
\begin{proof}
The first and second natural isomorphisms comes from the spectral sequence (for completed \'etale cohomology) discussed in \cite[Proposition 2.4.1]{Emint}, since localising at $\m$ kills $H^i_D(\up,\FFk)$ and $\widetilde{H}^i_D(\up,\FFk)$ for $i=0,2$. 

The same argument works for the third and fourth isomorphisms, since as noted in the proof of \cite[Proposition 2.4.1]{Emint} the spectral sequence of \cite[Corollary 2.2.18]{Emint} can be obtained using the Hochschild-Serre spectral sequence for \'etale cohomology, which applies equally well to the cohomology of the special fibre of $\mathbb{M}_\q(U_p\up)$. 

We now turn to the fifth isomorphism. By the first short exact sequence of Proposition \ref{prop:es} and the exactness of localising at $\m$, inverting $\varpi$ and then passing to locally analytic vectors, there is a short exact sequence
$$\minCDarrowwidth20pt\begin{CD}0@>>>\widetilde{H}^1_{\red}(\up,\FFk)^{an}_{\rhobar,\ep} @>>>\widetilde{H}^1_D(\up,\FFk)^{an}_{\rhobar,\ep}\\@.@>>>\widetilde{X}_\q(\up,\FFk)(-1)^{an}_{\rhobar,\ep}@>>>0.\end{CD}$$
Taking the long exact sequence of $\g$-cohomology then gives an exact sequence 
%$$\xymatrix{0\ar[r]&H^0(\g,\widetilde{H}^1_\red(\up,\FFk)^{an}_{\m,\ep}) \ar[r]&H^0(\g,\widetilde{H}^1_D(\up,\FFk)^{an}_{\m,\ep})\\\ar[r]&H^0(\g,\widetilde{X}_\q(\up,\FFk)(-1)^{an}_{\m,\ep})\ar[r]&H^1(\g,\widetilde{H}^1_\red(\up,\FFk)^{an}_{\m,\ep}),}$$ 
%
$$\minCDarrowwidth20pt\begin{CD}0@>>>H^0(\g,\widetilde{H}^1_\red(\up,\FFk)^{an}_{\rhobar,\ep})@>>>H^0(\g,\widetilde{H}^1_D(\up,\FFk)^{an}_{\rhobar,\ep})\\@>>>H^0(\g,\widetilde{X}_\q(\up,\FFk)(-1)^{an}_{\rhobar,\ep})@>>>H^1(\g,\widetilde{H}^1_\red(\up,\FFk)^{an}_{\rhobar,\ep}),\end{CD}$$
so applying the third, first and fourth part of the proposition to the first, second and fourth terms in this sequence respectively we get a short exact sequence
$$\minCDarrowwidth20pt\begin{CD}0@>>> H^1_{\red}(\up,\FFk)_{\rhobar,\ep} @>>>H^1_D(\up,\FFk)_{\rhobar,\ep}\\@. @>>>H^0(\g,\widetilde{X}_\q(\up,\FFk)(-1)^{an}_{\rhobar,\ep})@>>>0,\end{CD}$$

so identifying this with the short exact sequence of Proposition 2.3.2 gives the desired isomorphism $$H^0(\g,\widetilde{X}_\q(\up,\FFk)^{an}_{\rhobar,\ep})\cong X_\q (\up,\FFk)_{\rhobar,\ep}.$$

Finally the sixth isomorphism follows from the first four in the same way as the fifth, using the second short exact sequence of Proposition \ref{prop:es}.
\end{proof}

Given a continuous $G(\Q_p)$-representation $V$ we denote by $V^{alg}$ the space of locally algebraic vectors in $V$. Let $E(1)$ denote a one dimensional $E$-vector space on which $G(\Q_p) \times G_\q$ acts by $\prod_{v\mid p}N_{F_v/\Q_p}\circ (\det)^{-1} \otimes \epsilon$, where $\epsilon$ is the cyclotomic character of $G_\q$. For any $n \in \Z$ let $E(n)=E(1)^{\otimes n}$. We can now proceed as in Theorem 7.4.2 of \cite{Emlgc} to deduce
\begin{theorem}
There are natural $G(\Q_p) \times G_\q \times \T(\up)_{\rhobar}$-equivariant isomorphisms \begin{enumerate}\item $\bigoplus_{k,n\in \Z} H^1_D(\up,\FFk)_{\rhobar,\ep}\otimes_{\ep} W_{k,\pf}^\vee \otimes_{\ep} \ep(n) \cong \widetilde{H}^1_D(\up,\OO_\pf)^{alg}_{\rhobar,\ep}.$ 
\item $\bigoplus_{k,n\in \Z} H^1_\red(\up,\FFk)_{\rhobar,\ep}\otimes_{\ep} W_{k,\pf}^\vee \otimes_{\ep} \ep(n) \cong \widetilde{H}^1_\red(\up,\OO_\pf)^{alg}_{\rhobar,\ep}.$ 
\item $\bigoplus_{k,n\in \Z} H^1_\red(\up,\rFFk)_{\rhobar,\ep}\otimes_{\ep} W_{k,\pf}^\vee \otimes_{\ep} \ep(n) \cong \widetilde{H}^1_\red(\up,r_*r^*\OO_\pf)^{alg}_{\rhobar,\ep}.$ 
\item $\bigoplus_{k,n\in \Z} X_\q(\up,\FFk)_{\rhobar,\ep}\otimes_{\ep} W_{k,\pf}^\vee \otimes_{\ep} \ep(n) \cong \widetilde{X}_\q(\up,\OO_\pf)^{alg}_{\rhobar,\ep}.$ 
\item $\bigoplus_{k,n\in \Z} \check{X}_\q(\up,\FFk)_{\rhobar,\ep}\otimes_{\ep} W_{k,\pf}^\vee \otimes_{\ep} \ep(n) \cong \widetilde{\check{X}}_\q(\up,\OO_\pf)^{alg}_{\rhobar,\ep}.$ 
\end{enumerate}
In all the above direct sums, the index $k$ runs over $d$-tuples with all entries the same parity and $\ge 2$.
\end{theorem}
\begin{proof}
The proof is as for Theorem 7.4.2 of \cite{Emlgc}, since $$\{(\xi^{(k)})^\vee \otimes (\otimes_{i=1}^d (\det)^{-1} \circ \xi_i)^{n}\}_{k,n}$$ is a complete set of isomorphism class representatives of irreducible algebraic representations of $G$ which factor through $G^c$. Note that our definition of $E(1)$ is given in terms of $(\det)^{-1}$ --- the sign change from loc. cit. is due to the convention for the canonical model used by Carayol (and used here), where the Galois action on the group of connected components is the inverse of that used by Emerton.
\end{proof}

Applying the same arguments to $G'$, with $\m'$ a non-Eisenstein maximal ideal of $\T'(V^p)$ (and associated Galois representation $\rhobar'$) we also obtain
\begin{theorem}\label{Ylocalg}
There are natural $G'(\Q_p) \times G_\q \times \T'(V^p)_{\rhobar'}$-equivariant isomorphisms \begin{enumerate}\item $\bigoplus_{k,n\in \Z} H^1_{D'}(V^p,\FFk)_{\rhobar',\ep}\otimes_{\ep} W_{k,\pf}^\vee \otimes_{\ep} \ep(n) \cong \widetilde{H}^1_{D'}(V^p,\OO_\pf)^{alg}_{\rhobar',\ep}.$ 
\item $\bigoplus_{k,n\in \Z} Y_\q(V^p,\FFk)_{\rhobar',\ep}\otimes_{\ep} W_{k,\pf}^\vee \otimes_{\ep} \ep(n) \cong \widetilde{Y}_\q(V^p,\OO_\pf)^{alg}_{\rhobar',\ep}.$ 
\item $\bigoplus_{k,n\in \Z} \check{Y}_\q(V^p,\FFk)_{\rhobar',\ep}\otimes_{\ep} W_{k,\pf}^\vee \otimes_{\ep} \ep(n) \cong \widetilde{\check{Y}}_\q(V^p,\OO_\pf)^{alg}_{\rhobar',\ep}.$ 
\end{enumerate}
\end{theorem}

\subsection{Cofreeness results}
We now give analogues of Corollary 5.3.19 in \cite{Emlg}. 
\begin{proposition}\label{prop:injsmooth}
Suppose $U_p$ is small enough so that $U_p$ is pro-$p$ and $U_p\up$ is neat. Then for each $s > 0$, 
\begin{itemize}\item$H^1_D(\up,\OO_\pf/\varpi^s)_{\rhobar}$, \item$H^1_\red(\up,\OO_\pf/\varpi^s)_{\rhobar}$, \item$H^1_\red(\up,r_*r^* \OO_\pf/\varpi^s)_{\rhobar}$,\end{itemize} are injective as smooth representations of $\ql$ over $\OO_\pf/\varpi^s$.
\end{proposition}
In the statement of the proposition, we write $\ql$ to denote the quotient of $U_p$ by the closure in $U_p$ of the projection to $U_p$ of $F^\times \cap U_pU^p$.
\begin{proof}
We follow the proof of Proposition 5.3.15 in \cite{Emlg}, with modifications due to the presence of infinitely many global units $\OO_F^\times$. Let $L$ be any finitely generated smooth representation of $\ql=U_p/\overline{F^\times\cap U_pU^p(\q)}$ over $\OO_\pf/\varpi^s$, with Pontryagin dual $L^\vee := \Hom_{\OO_\pf/\varpi^s}(L,\OO_\pf/\varpi^s)$. For brevity we denote $\ql$ by $\overline{U}_p$. There is an induced local system $\mathscr{L}^\vee$ on each of the Shimura curves $M(U'_p\up)$ as $U'_p$ varies over the open normal subgroups of $U_p$. As in Proposition 5.3.15 of \cite{Emlg} there is a natural isomorphism $$H^1(M(U_p\up),\mathscr{L}^\vee)_{\rhobar} \cong \Hom_{\overline{U}_p}(L,H^1_D(\up,\OO_\pf/\varpi^s)_{\rhobar}).$$
Now starting from a short exact sequence $0\rightarrow L_0 \rightarrow L_1 \rightarrow L_2\rightarrow 0$ of finitely generated smooth $\overline{U}_p$-representations over $\OO_\pf/\varpi^s$ we obtain a short exact sequence of sheaves on $M(U_p\up)$:
$$\minCDarrowwidth10pt\begin{CD}0 @>>> \mathscr{L}^\vee_2 @>>> \mathscr{L}^\vee_1 @>>> \mathscr{L}^\vee_0 @>>> 0.\end{CD}$$
Taking the associated long exact cohomology sequence and localising at $\m$ gives another short exact sequence $$\minCDarrowwidth10pt\begin{CD}0 @>>> H^1(M(U_p\up),\mathscr{L}^\vee_2)_{\rhobar} @>>> H^1(M(U_p\up),\mathscr{L}^\vee_1)_{\rhobar} \\@.@>>> H^1(M(U_p\up),\mathscr{L}^\vee_0)_{\rhobar} @>>> 0,\end{CD}$$ or equivalently \small$$\minCDarrowwidth10pt\begin{CD}0 @>>> \Hom_{\overline{U}_p}(L_2,H^1_D(\up,\OO_\pf/\varpi^s)_{\rhobar}) @>>> \Hom_{\overline{U}_p}(L_1,H^1_D(\up,\OO_\pf/\varpi^s)_{\rhobar}) \\ @.@>>>\Hom_{\overline{U}_p}(L_0,H^1_D(\up,\OO_\pf/\varpi^s)_{\rhobar}) @>>> 0.\end{CD}$$\normalsize
Therefore we conclude that $H^1_D(\up,\OO_\pf/\varpi^s)_{\rhobar}$ is injective as a smooth representation of $\overline{U}_p$ over $\OO_\pf/\varpi^s$ (any smooth representation is a direct limit of finitely generated smooth representations, so exactness of a functor on the latter implies exactness on the former). The same argument applies to the other cohomology spaces (using Lemma \ref{oldred} for the final case).
\end{proof}
\begin{definition}
We say that a topological representation $V$ of a topological group $\Gamma$ over $\OO_\pf$ is \emph{cofree} if there is a topological isomorphism of representations $V \cong \mathscr{C}(\Gamma,\OO_\pf)^r$ for some integer $r$, where $\mathscr{C}(\Gamma,\OO_\pf)$ denotes the space of continuous functions from $\Gamma$ to $\OO_\pf$ with the right regular action of $\Gamma$.
\end{definition}
\begin{corollary}\label{cor:cofree1}
If $U_p$ is small enough so that $U_p$ is pro-$p$ and $U_p\up$ is neat, then $\widetilde{H}^1_D(\up,\OO_\pf)_{\rhobar}$, $\widetilde{H}^1_\red(\up,\OO_\pf)_{\rhobar}$ and $\widetilde{H}^1_\red(\up,r_*r^*\OO_\pf)_{\rhobar}$ are all cofree representations of $\ql$.
\end{corollary}
\begin{proof}
The proof is as for Corollary 5.3.19 in \cite{Emlg}. We will give the proof just for $\widetilde{H}^1_D(\up,\OO_\pf)_{\rhobar}$, the other cases being proved in exactly the same way. We again write $\overline{U}_p$ for $\ql$. It suffices to show that there is an isomorphism of smooth representations of $\overline{U}_p$ over $\OO_\pf$
$$\widetilde{H}^1_D(\up,\OO_\pf)_{\rhobar}/\varpi^s \widetilde{H}^1_D(\up,\OO_\pf)_{\rhobar} \cong \mathscr{C}(\overline{U}_p ,\OO_\pf/\varpi^s\OO_\pf)^r,$$ for some $r > 0$ and each $s > 0$ (since then $r$ is independent of $s$ and we may pass to the projective limit in $s$). Since $U_p$ is pro-$p$, the quotient $\overline{U}_p$ is pro-$p$. It follows that the completed group ring $(\OO_\pf/\varpi^s\OO_\pf)[[\overline{U}_p]]$ is a non-commutative local ring. Hence a non-zero finitely generated projective  $(\OO_\pf/\varpi^s\OO_\pf)[[\overline{U}_p]]$-module is in fact a free module. Dualising, if a smooth admissible $\overline{U}_p$-representation over $\OO_\pf/\varpi^s\OO_\pf$ is injective as a smooth representation, it is isomorphic to $\mathscr{C}(\overline{U}_p ,\OO_\pf/\varpi^s\OO_\pf)^r$ for some $r > 0$. Hence it suffices to show that $\widetilde{H}^1_D(\up,\OO_\pf)_{\rhobar}/\varpi^s \widetilde{H}^1_D(\up,\OO_\pf)_{\rhobar}$ is injective as a smooth representation of $\overline{U}_p$ over $\OO_\pf/\varpi^s\OO_\pf$. Since $\m$ is non-Eisenstein this is equivalent to injectivity of $H^1_D(\up,\OO_\pf/\varpi^s)_{\rhobar}$, which follows from Proposition \ref{prop:injsmooth}.
\end{proof}
\begin{corollary}\label{cor:cofree2}
If $U_p$ is small enough so that $U_p$ is pro-$p$ and $U_p\up$ is neat, then $\widetilde{X}_\q(\up,\OO_\pf)_{\rhobar}$ is a cofree representation of $\ql$.
\end{corollary}
\begin{proof}
Taking the first exact sequence of Proposition \ref{prop:es} and applying the functor \linebreak[4]$\Hom_{cts}(-,\OO_\pf)$ gives a short exact sequence (we can ignore the Tate twist for the purposes of this Corollary)
\begin{align*}0\rightarrow \Hom_{cts}(\widetilde{X}_\q(\up,\OO_\pf)_{\rhobar},\OO_\pf) &\rightarrow \Hom_{cts}(\widetilde{H}^1_D(\up,\OO_\pf)_{\rhobar},\OO_\pf)\\ &\rightarrow \Hom_{cts}(\widetilde{H}^1_{\red}(\up,\OO_\pf)_{\rhobar},\OO_\pf)\rightarrow 0.\end{align*}
Corollary \ref{cor:cofree1} implies that the second and third non-zero terms in this sequence are free $\OO_\pf[[\ql]]$-modules, so $\Hom_{cts}(\widetilde{X}_\q(\up,\OO_\pf)_{\rhobar},\OO_\pf)$ is projective, hence free since $\OO_\pf[[\ql]]$ is local. Applying the functor $\Hom_{cts}(-,\OO_\pf)$ again gives the desired result.
\end{proof}
Similarly we obtain
\begin{corollary}\label{cor:cofree3}
If $V_p$ is small enough so that $V_p$ is pro-$p$ and $V_pV^p$ is neat, then \begin{itemize}\item $\widetilde{H}^1_{D'}(V^p,\OO_\pf)_{\rhobar'}$, \item $\widetilde{Y}_\q(V^p,\OO_\pf)_{\rhobar'}$, \item $\widetilde{\check{Y}}_\q(V^p,\OO_\pf)_{\rhobar'}$,\end{itemize} are cofree representations of $\qlV$.
\end{corollary}
\subsection{Jacquet-Langlands maps between eigenvarieties}
We can now apply the results of section \ref{pJL} to deduce some cases of \emph{overconvergent} Jacquet-Langlands functoriality, or in other words maps between eigenvarieties interpolating the classical Jacquet-Langlands correspondence. Let $\widehat{T}$ denote the rigid analytic variety parameterising continuous characters $\chi: T \rightarrow \C_p^\times$.
\begin{definition}\label{def:evar}
Suppose $V$ is a $\varpi$-adically admissible $G(\Q_p)$-representation over $\OO_\pf$, with a commuting $\OO_\pf$-linear action of a commutative Noetherian $\OO_\pf$-algebra $\mathbf{A}$. The essentially admissible locally analytic $T$-representation $J_B(V_E)$ gives rise (by duality) to a coherent sheaf $\mathscr{M}$ on $\widehat{T}$, and the action of $\mathbf{A}$ gives rise to a coherent sheaf on $\widehat{T}$ of $\OO_\pf$-algebras $\mathscr{A} \hookrightarrow End_{\widehat{T}}(\mathscr{M})$.
Define the \emph{eigenvariety} $$\mathscr{E}(V,\mathbf{A})$$ to be the rigid analytic space given by taking the relative spectrum of $\mathscr{A}$ over $\widehat{T}$.
\end{definition}

\begin{lemma}
Suppose $V$ and $\mathbf{A}$ are as in Definition \ref{def:evar}, further suppose that $V$ is a (non-zero) cofree representation of $U_p/X$ for some compact open subgroup $U_p$ of $G(\Q_p)$, $X$ some closed subgroup of $Z(\Q_p) \cap U_p$. Denote by $T_0$ the intersection $T \cap U_p$. Then $\mathscr{E}(V,\mathbf{A})$ is equidimensional of dimension equal to the dimension of the rigid analytic variety $\widehat{T_0/X}$ parameterising continuous characters of $T_0$ which are trivial on $X$.
\end{lemma}
\begin{proof}
This is a combination of a mild generalisation of the proof of \cite[Proposition 4.2.36]{MR2292633} and \cite[Corollary 4.1]{cclr}. We will outline the argument. The cofreeness assumption on $V$ implies that $V_E^{an}$ is isomorphic to $\mathscr{C}^{an}(U_p/X,E)^r$, where $\mathscr{C}^{an}(U_p/X,E)$ denotes the space of locally analytic functions from $U_p/X$ to $E$. We write $M$ for the strong dual of $J_B(V_E)$, which is equal to the space of global sections $\mathscr{M}(\widehat{T})$. $M$ is a module for the Fr\'echet algebra of locally analytic functions $\mathscr{C}^{an}(\widehat{T_0/X},E)$. We then write $\widehat{T_0/X}$ as a union of admissible affinoid subdomains $\mathrm{MaxSpec}(A_n)$ and apply the argument of \cite[Corollary 4.1]{cclr} to $M\widehat{\otimes}_{\mathscr{C}^{an}(\widehat{T_0/X},E)} A_n$ which is the finite slope part of an orthonormalisable Banach $A_n$-module with respect to some compact operator.
\end{proof}
In our applications we will have $X$ equal to the closure in $U_p$ of the $p$-factor of $F^\times \cap U_pU^p$ for some tame level $U^p$. The dimension of $\widehat{T_0/X}$ will therefore depend on the defect in Leopoldt's conjecture for $F$ and $p$. To be precise, it will equal $$\dim(\widehat{T_0})- (d-1 - \delta)=d+1+\delta,$$ where $\delta$ is the defect in Leopoldt's conjecture.
Applying the above lemma to the cofreeness results of the previous section, we have the following:
\begin{corollary}\label{lotsequidim}
The following eigenvarieties are all equidimensional of dimension equal to $d+1+\delta$:
\begin{itemize}
\item $\mathscr{E}(\widetilde{H}^1_D(U^p,\OO_\pf)_{\rhobar},\T(U^p)_{\rhobar})$
\item $\mathscr{E}(\widetilde{X}_\q(\up,\OO_\pf)_{\rhobar},\T(\up)_{\rhobar})$
\item $\mathscr{E}(\widetilde{H}_{D'}^1(V^p,\OO_\pf)_{\rhobar'},\T'(V^p)_{\rhobar'})$
\item $\mathscr{E}(\widetilde{Y}_\q(V^p,\OO_\pf)_{\rhobar'},\T'(V^p)_{\rhobar'})$
\item $\mathscr{E}(\widetilde{\check{Y}}_\q(V^p,\OO_\pf)_{\rhobar'},\T'(V^p)_{\rhobar'})$
\end{itemize}
\end{corollary}
\begin{remark}
Since we are constructing our eigenvarieties using Hecke algebras deprived of the Hecke operators at places where the tame level is non-trivial, all the above eigenvarieties are reduced.
\end{remark}
\begin{lemma}\label{lnew}
The map $\widetilde{X}_\q(\up,\OO_\pf)_{\rhobar} \rightarrow \widetilde{H}^1_D(\up,\OO_\pf)_{\rhobar}$ induced by composing the monodromy pairing embedding $\widetilde{X}_\q(\up,\OO_\pf)_{\rhobar} \rightarrow \widetilde{\check{X}}_\q(\up,\OO_\pf)_{\rhobar}$ with the embedding $\widetilde{\check{X}}_\q(\up,\OO_\pf)_{\rhobar} \rightarrow \widetilde{H}^1_D(\up,\OO_\pf)_{\rhobar}$ provided by Proposition \ref{prop:es} induces an isomorphism of eigenvarieties
$$\mathscr{E}(\widetilde{X}_\q(\up,\OO_\pf)_{\rhobar},\T(\up)_{\rhobar}) \cong \mathscr{E}(\widetilde{H}^1_D(\up,\OO_\pf)^{\q\mhyphen\mathrm{new}}_{\rhobar},\T(\up)_{\rhobar}).$$
\end{lemma}
\begin{proof}
This follows from Proposition \ref{prop:Xnew}.
\end{proof}
We have a similar result for the group $G'$.
\begin{lemma}\label{lemmaYH1}
The map $\widetilde{Y}_\q(V^p,\OO_\pf)_{\rhobar'} \rightarrow \widetilde{H}^1_{D'}(V^p,\OO_\pf)_{\rhobar'}$ induced by composing the monodromy pairing embedding $\widetilde{Y}_\q(V^p,\OO_\pf)_{\rhobar'} \rightarrow \widetilde{\check{Y}}_\q(V^p,\OO_\pf)_{\rhobar'}$ with the embedding \\$\widetilde{\check{Y}}_\q(V^p,\OO_\pf)_{\rhobar'} \rightarrow \widetilde{H}_{D'}^1(V^p,\OO_\pf)_{\rhobar'}$ provided by Proposition \ref{prop:Yses} induces an isomorphism between eigenvarieties
$$\mathscr{E}(\widetilde{Y}_\q(V^p,\OO_\pf)_{\rhobar'},\T'(V^p)_{\rhobar'}) \cong \mathscr{E}(\widetilde{H}_{D'}^1(V^p,\OO_\pf)_{\rhobar'},\T'(V^p)_{\rhobar'}).$$
\end{lemma}
\begin{proof}
This follows from Proposition \ref{YHsame}.
\end{proof}

\begin{lemma}\label{lemmaXnew}
The map $\widetilde{X}_\qb(U^p(\qa\qb),\OO_\pf)_{\rhobar}^{\qa\mhyphen\mathrm{new}}\rightarrow \widetilde{H}^1_D(U^p(\qa\qb),\OO_\pf)_{\rhobar}^{\qa\qb\mhyphen\mathrm{new}}$ induces an isomorphism between eigenvarieties
\small$$\mathscr{E}(\widetilde{X}_\qb(U^p(\qa\qb),\OO_\pf)_{\rhobar}^{\qa\mhyphen\mathrm{new}},\T(U^p(\qa\qb))_{\rhobar})\cong \mathscr{E}(\widetilde{H}^1_D(U^p(\qa\qb),\OO_\pf)_{\rhobar}^{\qa\qb\mhyphen\mathrm{new}},\T(U^p(\qa\qb))_{\rhobar}).$$\normalsize
\end{lemma}
\begin{proof}
This follows from Proposition \ref{prop:Xdoublenew}.
\end{proof}
We now put ourselves in the situation of section \ref{pJL}, so $U^p$ and $V^p$ are isomorphic at places away from $\qa$ and $\qb$, and $\m$ and $\m'$ give rise to the same irreducible mod $p$ Galois representation $\rhobar$. 
\begin{theorem}\emph{(Overconvergent Jacquet-Langlands correspondence)}\label{ocjl}
The map $$\widetilde{Y}_\qa(V^p,\OO_\pf)_{\rhobar} \rightarrow \widetilde{X}_\qb(U^p(\qa\qb),\OO_\pf)_{\rhobar}$$ given by Theorem \ref{RRES} induces an isomorphism between eigenvarieties
$$\mathscr{E}(\widetilde{H}^1_{D'}(V^p,\OO_\pf)_{\rhobar},\T'(V^p)_{\rhobar}) \cong \mathscr{E}(\widetilde{H}^1_D(U^p(\qa\qb),\OO_\pf)_{\rhobar}^{\qa\qb\mhyphen\mathrm{new}},\T(U^p(\qa\qb))_{\rhobar}).$$
\end{theorem}
\begin{proof}
It is clear that Theorem \ref{RRES} induces an isomorphism \small$\mathscr{E}(\widetilde{Y}_\qa(V^p,\OO_\pf)_{\rhobar} ,\T'(V^p)_{\rhobar}) \cong \mathscr{E}(\widetilde{X}_\qb(U^p(\qa\qb),\OO_\pf)_{\rhobar}^{\qa\mhyphen\mathrm{new}},\T(U^p(\qa\qb))_{\rhobar})$\normalsize. Our theorem now follows by applying Lemmas \ref{lemmaYH1} and \ref{lemmaXnew}.
\end{proof}

Finally we note that level raising results in the same spirit as those in \cite{chicomp,cclr} follow from the equidimensionality of the eigenvarieties described in this section. In particular we have
\begin{corollary}\label{weirdlr}
The following eigenvarieties are equidimensional (of dimension $d+1+\delta$):
\begin{itemize}
\item $\mathscr{E}(\widetilde{H}^1_{D}(U^p(\q),\OO_\pf)_{\rhobar}^{\q\mhyphen\mathrm{new}},\T(U^p(\q))_{\rhobar})$
\item $\mathscr{E}(\widetilde{H}^1_D(U^p(\qa\qb),\OO_\pf)_{\rhobar}^{\qa\qb\mhyphen\mathrm{new}},\T(U^p(\qa\qb))_{\rhobar})$
\end{itemize}
\end{corollary}
\begin{proof}
The first claim follows from the second part of Corollary \ref{lotsequidim} and Lemma \ref{lnew}. The second claim follows from the fourth part of Corollary \ref{lotsequidim} and Theorem \ref{ocjl}.\end{proof}
Recall the natural maps $$i_1:\check{X}_\qb(U(\qb),\OO_\pf)_{\rhobar}\rightarrow\check{X}_\qb(U(\qa\qb),\OO_\pf)_{\rhobar}$$ appearing in Section \ref{sec:RR}, which induce an injection $$i_1:\widetilde{\check{X}}_\qb(U^p(\qb),\OO_\pf)_{\rhobar} \rightarrow\widetilde{\check{X}}_\qb(U^p(\qa\qb),\OO_\pf)_{\rhobar}$$ by Theorem \ref{RRES}. This induces an embedding $i_1$ from $$\mathscr{E}(\widetilde{H}^1_{D}(U^p(\qb),\OO_\pf)_{\rhobar}^{\qb\mhyphen\mathrm{new}},\T(U^p(\qb))_{\rhobar})$$ to $$\mathscr{E}(\widetilde{H}^1_D(U^p(\qa\qb),\OO_\pf)_{\rhobar}^{\qb\mhyphen\mathrm{new}},\T(U^p(\qa\qb))_{\rhobar}).$$ The following corollary characterises the intersection of the image of $i_1$ (the $\qa$-old points of the eigenvariety) with $$\mathscr{E}(\widetilde{H}^1_{D}(U^p(\qa\qb),\OO_\pf)_{\rhobar}^{\qa\qb\mhyphen\mathrm{new}},\T(U^p(\qa\qb))_{\rhobar}).$$
\begin{corollary}\label{levelraising}
Suppose we have a point $$x \in \mathscr{E}(\widetilde{H}^1_{D}(U^p(\qb),\OO_\pf)_{\rhobar}^{\qb\mhyphen\mathrm{new}},\T(U^p(\qb))_{\rhobar}).$$ Then $$i_1(x) \in \mathscr{E}(\widetilde{H}^1_{D}(U^p(\qa\qb),\OO_\pf)_{\rhobar}^{\qa\qb\mhyphen\mathrm{new}},\T(U^p(\qa\qb))_{\rhobar})$$ if and only if the $T_\qa$ and $S_\qa$ eigenvalues of $x$ satisfy $$T_\qa(x)^2-(\mathbf{N}\qa+1)^2S_\qa(x)=0.$$
\end{corollary}
\begin{remark}
Note that if the conditions of this Corollary are satisfied, then Corollary \ref{weirdlr} implies that $i_1(x)$ lies in an equidimensional family of newforms, with `full' dimension.
\end{remark}
\begin{proof}
A standard calculation shows that the composite $i^\dagger\circ i$ acts by the matrix $$\begin{pmatrix}
\mathbf{N}\qa+1 & T_\qa \\ S_\qa^{-1}T_\qa & \mathbf{N}\qa+1
\end{pmatrix}$$ on $(\widetilde{\check{X}}_\qb(U^p(\qb),E)_{\rhobar}^{an})^{\oplus 2}$. Write $\widetilde{x}$ for a class in $\widetilde{\check{X}}_\qb(U^p(\qb),E')_{\rhobar}^{an}$ giving rise to the point $x$ (where $E'/E$ is a field over which $x$ is defined). Now if the criterion $T_\qa(x)^2-(\mathbf{N}\qa+1)^2S_\qa(x)=0$ is satisfied there is a non-zero element $\widetilde{y}$ of $$E'\cdot \widetilde{x} \oplus E'\cdot \widetilde{x} \subset (\widetilde{\check{X}}_\qb(U^p(\qb),E')_{\rhobar}^{an})^{\oplus 2}$$ in the kernel of $i^\dagger\circ i$ --- namely, $$((\mathbf{N}\qa+1)S_\qa(x)\cdot \widetilde{x},-T_\qa(x)\cdot \widetilde{x}).$$ The element $i(\widetilde{y})$ then gives rise to the point $$i_1(x)\in\mathscr{E}(\widetilde{\check{X}}_\qb(U^p(\qa\qb),\OO_\pf)_{\rhobar}^{\qa\mhyphen\mathrm{new}},\T(U^p(\qa\qb))_{\rhobar}),$$ so we are done by Lemmas \ref{monodromy} and \ref{lemmaXnew}. Conversely if $$i_1(x) \in \mathscr{E}(\widetilde{H}^1_{D}(U^p(\qa\qb),\OO_\pf)_{\rhobar}^{\qa\qb\mhyphen\mathrm{new}},\T(U^p(\qa\qb))_{\rhobar})$$ then the local Galois representation of $G_\qa$ attached to $x$ must be a sum of two unramified characters $\chi_1,\chi_2$ with $\chi_1(Frob_\qa)/\chi_2(Frob_\qa)=(\mathbf{N}\qa)^{\pm 1}$, since it is unramified and $i_1(x)$ lies in a family of (twists of) classical $\qa$-newforms. This implies that $T_\qa(x)^2-(\mathbf{N}\qa+1)^2S_\qa(x)=0$.
\end{proof}
\begin{remark}
We would also like to prove a level raising result characterising the intersection of $\q$-old and $\q$-new forms with tame level $U^p(\q)$. Corollary \ref{weirdlr} would be enough to deduce such a result, provided one shows that the map $$i:(\widetilde{H}^1_{D}(U^p,E)_{\rhobar}^{an})^{\oplus 2} \rightarrow \widetilde{H}^1_{D}(U^p(\q),E)_{\rhobar}^{an}$$ is an injection. This would follow from injectivity of the map $$i:\widetilde{H}^1_{D}(U^p,\OO_\pf)_{\rhobar}^{\oplus 2} \rightarrow \widetilde{H}^1_{D}(U^p(\q),\OO_\pf)_{\rhobar},$$ but this is equivalent to showing the injectivity of $$i:H^1(M(U_pU^p)_{\overline{F}},\OO_\pf/\varpi)_{\rhobar}^{\oplus 2} \rightarrow H^1(M(U_pU^p(\q))_{\overline{F}},\OO_\pf/\varpi)_{\rhobar}$$ for all compact open subgroups $U_p$ of $G(\Q_p)$ --- this version of Ihara's lemma is not known in this generality (as far as the author is aware). The approach of \cite{DT} (which is extended to Shimura curves over totally real fields in \cite{Cheng-preprint}), via crystalline methods, can work only when $U_p$ is maximal compact (since one needs $M(U_pU^p(\q))_{\overline{F}}$ to have good reduction at places dividing $p$). 
\end{remark}

\section*{acknowledgements}
This paper is based on part of the author's PhD thesis, written under the supervision of Kevin Buzzard, to whom I am grateful for providing such excellent guidance. This research crucially relies on work of Matthew Emerton, whom I also thank for several helpful conversations and communicating details of an earlier draft of \cite{Emlg}. The author was supported by an Engineering and Physical Sciences Research Council doctoral training grant during the bulk of the time spent researching the contents of this paper. The writing up process was completed whilst the author was a member of the Institute for Advanced Study, supported by National Science Foundation grant DMS-0635607. The author is currently supported by Trinity College, Cambridge. I thank all these institutions for their support. I would also like to thank the anonymous referee for their helpful comments.

\end{document}